\documentclass{ieeetj}
\usepackage{cite}
\usepackage{amsmath,amssymb,amsfonts}
\usepackage{textcomp}
\usepackage{xcolor}
\usepackage{hyperref}
\hypersetup{hidelinks=true}

\def\BibTeX{{\rm B\kern-.05em{\sc i\kern-.025em b}\kern-.08em
    T\kern-.1667em\lower.7ex\hbox{E}\kern-.125emX}}
\AtBeginDocument{\definecolor{tmlcncolor}{cmyk}{0.93,0.59,0.15,0.02}\definecolor{NavyBlue}{RGB}{0,86,125}}

\usepackage{subfigure}
\usepackage{ascmac}
\usepackage[dvips]{graphicx}
\usepackage{psfrag}
\usepackage{amsbsy}
\usepackage{latexsym}
\usepackage{color}
\usepackage{theorem}

\usepackage{version}

\graphicspath{{./images/}}

\newtheorem{proposition}{Proposition}

\newtheorem{definition}{Definition}
\newtheorem{theorem}{Theorem}
\newtheorem{lemma}{Lemma}
\newtheorem{corollary}{Corollary}

\newtheorem{fact}{Fact}
\newtheorem{remark}{Remark}

\newcommand{\euclidspace}{{\mathcal{H}}}
\newcommand{\signal}[1]{{\boldsymbol{#1}}}
\newcommand{\Natural}{{\mathbb N}}
\newcommand{\norm}[1]{\left\|#1\right\|}
\newcommand{\abs}[1]{\left|#1\right|}

\newcommand{\real}{{\mathbb R}}

\newcommand{\innerprod}[2]{\left\langle{#1},{#2}\right\rangle}

\newcommand{\argmin}{\operatornamewithlimits{argmin}}

\newcommand{\gzh}{{\Gamma_0(\mathcal{H})}}
\newcommand{\prox}{\mathbf{Prox}}
\newcommand{\sprox}{{\rm s\mbox{-}Prox}}
\newcommand{\dom}{{\rm dom}~}
\newcommand{\range}{{\rm range}~}
\newcommand{\cont}{{\rm cont}~}

\newcommand{\id}{{\rm Id}}

\def\OJlogo{\vspace{-4pt}}
\def\seclogo{\vspace{10pt}}

\def\authorrefmark#1{\ensuremath{^{\textbf{#1}}}}

\begin{document}

\doiinfo{10.1109/OJSP.2025.3579646}

\markboth{}{Author {et al.}}

\title{Continuous Relaxation of Discontinuous Shrinkage Operator:
Proximal Inclusion and Conversion}

\author{Masahiro Yukawa\authorrefmark{1}, Senior Member, IEEE}
\affil{Department of Electronics and Electrical Engineering, Keio
University, Yokohama, JAPAN}
\corresp{Corresponding author: Masahiro Yukawa (email: yukawa@elec.keio.ac.jp).}
\authornote{This work was partially supported by JSPS Grants-in-Aid
(23K22762).}

\begin{abstract}
We present a principled way of deriving
a continuous relaxation of a given discontinuous shrinkage operator,
which is based on two fundamental results, proximal inclusion and conversion.
Using our results, the discontinuous operator is converted, via double
 inversion, to a continuous operator; more precisely, the associated
 ``set-valued'' operator is converted to a ``single-valued'' Lipschitz
 continuous operator.
The first illustrative example is the firm shrinkage operator which can
be derived as a continuous relaxation of the hard shrinkage operator.
We also derive a new operator as a continuous relaxation of 
the discontinuous shrinkage operator associated with 
the so-called reversely ordered weighted $\ell_1$ (ROWL) penalty.
Numerical examples demonstrate potential advantages of the continuous relaxation.

\end{abstract}



\maketitle

\section{Introduction}\label{sec:intro}


Triggered by the seminal work \cite{donoho94,donoho95_asymptopia} by
Donoho and Johnstone in 1994, 
the soft and hard shrinkage operators have been studied
in the context of sparse modeling \cite{donoho95,blumensath08},
which has a wide range of signal processing  applications
such as magnetic resonance imaging, radar, sparse coding, 
compressed sensing/sampling, signal dimensionality reduction, 
to name just a few \cite{foucart13}.
The soft shrinkage operator is characterized as the proximity operator
of the $\ell_1$ norm,
where the proximity operator of a convex function is nonexpansive
(Lipschitz continuous with unity constant) in general.
Because of this, it has widely been used
in the operator splitting algorithms
\cite{combettes05,bauschke_book19,condat_siam23}.
The hard shrinkage operator, on the other hand, is ``discontinuous'', 
while it causes no extra bias
in the estimates of large-magnitude coefficients.
Hard shrinkage is derived from the {\em set-valued} proximity operator
of the $\ell_0$ pseudo-norm.

The firm shrinkage operator \cite{gao97} has been proposed in 1997,
possessing a good balance
in the sense of (i) being Lipschitz continuous 
(with constant strictly greater than unity) and
(ii) significantly reducing the estimation biases.
It is the (generalized) proximity operator of 
the minimax concave (MC) penalty \cite{zhang,selesnick}
which is weakly convex.
Owing to this property, the proximal forward-backward splitting algorithm
employing firm shrinkage has a guarantee of convergence to 
a global minimizer of the objective function \cite{bayram16,yukawa_molgrad24}.
Another discontinuous shrinkage operator
has been proposed in the context of image restoration
based on the reversely ordered weighted $\ell_1$ (ROWL) penalty 
\cite{sasaki24},
which gives small (possibly zero) weights to dominant coefficients
to avoid underestimation
so that edge sharpness and gradation smoothness of images are enhanced
simultaneously.
The ROWL shrinkage operator requires no knowledge about the ``magnitude'' of
the dominant coefficients but instead it requires the ``number'' of
those dominant components essentially.
This is in sharp contrast to firm shrinkage which 
requires (at least a rough approximation, or a sort of bound of) 
the ``magnitude'' of the dominant coefficients
but it does not care about the ``number''.
This implies that ROWL shrinkage would be preferable in such situations
where the magnitude of dominant components tends to change but the
number (or at least its rough estimate) can be assumed to be known {\it
a priori}.
Despite this potential benefit,
the ROWL shrinkage operator is discontinuous on the boundaries
where some of the components share the same magnitude,
provided that the corresponding weights are different.
A natural question would be the following: {\it is it possible to
convert the discontinuous operator to a continuous one?}

To address the above question, our arguments in the present study rely on
our recent result \cite{yukawa_molgrad24}:
a given operator is the ``single-valued'' proximity operator of 
a $\eta$-weakly convex function for $\eta\in(0,1)$
if and only if
it is a monotone Lipschitz-continuous gradient (MoL-Grad) denoiser
(i.e., it can be expressed as a Lipschitz-continuous gradient operator
of a differentiable convex function).
See Fact \ref{fact:weaklyconvex_necsuffcondition}.
It has been shown in \cite{yukawa_molgrad24} that 
the operator-regularization approach (or the plug-and-play method
\cite{venkatakrishnan13}) employing a MoL-Grad denoiser actually solves
a variational problem involving a weakly convex regularizer which is
characterized explicitly by the denoiser.

In this paper, we present a pair of fundamental findings concerning
the ``set-valued'' proximity operator\footnote{The ``set-valued'' proximity
operator has been studied previously in the literature
\cite{gribonval20,bauschke21}.
See also \cite{rockafellar10}.
} of (proper) nonconvex function, aiming to build a way of 
converting a discontinuous operator to its continuous relaxation.
First, given a nonconvex function,
the image of a point 
with respect to the proximity operator
is included by its image with respect to the proximity operator of
its lower-semicontinuous 1-weakly-convex envelope
(Theorem \ref{theorem:prox_inclusion}).
This well explains the known fact that hard shrinkage can also be
derived from the proximity operator of a certain weakly convex function
\cite{kowalski_icip14,bayram15}, as elaborated in Section 
\ref{sec:application}-\ref{subsec:l0_mc}.
Second, the ``set-valued'' proximity operator of a (proper) 
lower-semicontinuous 1-weakly-convex function can be converted,
via double inversion, to a MoL-Grad denoiser
(Theorem \ref{theorem:prox_conversion}).
Those proximal {\em inclusion} and {\em conversion} lead to a principled
way of deriving a continuous relaxation (a Lipschitz-continuous
relaxation, more specifically) of a given discontinuous
operator.
As an illustrative example, we show that the firm shrinkage operator
is obtained as a continuous relaxation of the hard shrinkage operator.
Under the same principle, we derive a continuous relaxation
of the ROWL shrinkage operator.
Numerical examples show that the continuous relaxation has 
potential advantages over the original discontinuous shrinkage.


\section{Preliminaries}
\label{sec:preliminaries}

Let $(\euclidspace,\innerprod{\cdot}{\cdot})$ be a real Hilbert space
with the induced norm $\norm{\cdot}$.
Let $\real$, $\real_+$, $\real_{++}$, and $\Natural$ denote
the sets of real numbers, nonnegative real numbers,
strictly positive real numbers, and nonnegative integers, respectively.


\subsection{Lipschitz Continuity of Operator, Set-valued Operator}
\label{subsec:nonexpansive}

Let
$\id:\euclidspace\rightarrow\euclidspace:x\mapsto x$ denote 
the identity operator on $\euclidspace$.
An operator $T:\euclidspace\rightarrow \euclidspace$ is
Lipschitz continuous
with constant $\kappa>0$ (or $\kappa$-Lipschitz continuous for short)
if
\begin{equation}
\norm{T(x)-T(y)}\leq \kappa \norm{x-y}, ~\forall (x,y)\in\euclidspace^2.
\label{eq:def_lipschitz}
\end{equation}

Let $2^{\euclidspace}$ denote the power set 
(the family of all subsets) of $\euclidspace$.
An operator $\mathsf{T}:\euclidspace\rightarrow 2^{\euclidspace}$ is called 
a set-valued operator, where $\mathsf{T}(x)\subset \euclidspace$ for every $x\in\euclidspace$.
Given a set-valued operator $\mathsf{T}$,
a mapping $U:\euclidspace\rightarrow \euclidspace$ 
such that $U(x)\in \mathsf{T}(x)$ for every $ x\in\euclidspace$
is called 
{\em a selection} of $\mathsf{T}$.
The inverse of a set-valued operator $\mathsf{T}$ is defined by
$\mathsf{T}^{-1}:\euclidspace\rightarrow 2^{\euclidspace}:
y \mapsto \{x\in\euclidspace\mid y\in \mathsf{T}(x)\}$,
which is again a set-valued operator in general.

An operator $\mathsf{T}:\euclidspace\rightarrow 2^{\euclidspace}$ is {\em monotone} if
\begin{equation}
\hspace*{-2em} \innerprod{x \! - \! y}{u \! - \! v}\geq 0,~\forall (x,u)\in{\rm gra~}\mathsf{T},~\forall
  (y,v)\in{\rm gra~}\mathsf{T},
\label{eq:def_monotone}
\end{equation}
where gra $\mathsf{T}:=\{(x,u)\in\euclidspace^2\mid u\in
\mathsf{T}(x)\}$ is the graph of $\mathsf{T}$.
The following fact is known \cite[Proposition 20.10]{combettes}.
\begin{fact}[Preservation of monotonicity]
\label{fact:monotone} 
Let $\mathcal{K}$ be a real Hilbert space,
$A:\euclidspace\rightarrow 2^\euclidspace$ and
$B:\mathcal{K}\rightarrow 2^\mathcal{K}$
be monotone operators,
$L:\euclidspace\rightarrow \mathcal{K}$
be a bounded linear operator,
and $c\in\real_+$ be a nonnegative constant.
Then, the operators
$A^{-1}$, $c A$, and $A+ L^* BL$ are monotone,
where $L^*$ denotes the adjoint operator of $L$.
\end{fact}

A monotone operator $\mathsf{T}$ is
{\em maximally monotone} if no other monotone operator
has its graph containing gra $\mathsf{T}$ properly.


\subsection{Lower Semicontinuity and Convexity of Function}
\label{subsec:properness}

A function $f:\euclidspace\rightarrow (-\infty,+\infty]:=
\real\cup \{+\infty\}$ is {\em proper} if the domain is nonempty; i.e.,
$\dom f:= \{x\in\euclidspace\mid f(x)<+\infty\}\neq \varnothing$.
A function $f$ is {\it lower semicontinuous (l.s.c.)} on $\euclidspace$ 
if the level set
${\rm lev}_{\leq a} f:=
\left\{x\in\euclidspace: f(x)\leq a\right\}$
is closed for every $a\in\real$.
A function $f:\euclidspace\rightarrow
(-\infty,+\infty]$ is convex on $\euclidspace$ if
$f(\alpha x + (1-\alpha)y)\leq 
\alpha f(x) + (1-\alpha)f(y)$ for every
$(x,y,\alpha)\in\dom f\times\dom f\times [0,1]$.
 For $\eta\in (-\infty,+\infty]$,
we say that $f$ is $\eta$-weakly convex
if $f+ (\eta/2) \norm{\cdot}^2$ is convex.
Here, $\eta<0$ (i.e., $f - (\abs{\eta}/2) \norm{\cdot}^2$ is convex)
means that $f$ is  $\abs{\eta}$-strongly convex.
Clearly, $\eta$-weak convexity of $f$ implies 
$\tilde{\eta}$-weak convexity of $f$ for an arbitrary $\tilde{\eta} \geq
\eta$.
When the ``minimal'' weak-convexity parameter is $\eta=+\infty$,
$f + \alpha\norm{\cdot}^2$ is nonconvex for every $\alpha\in \real$,
i.e., $f$ is not weakly convex (for any parameter in $\real$).
%
The set of all proper l.s.c.~convex functions
$f:\euclidspace\rightarrow (-\infty,+\infty]$
is denoted by $\Gamma_0(\euclidspace)$.

Given a proper function $f:\euclidspace\rightarrow (-\infty,+\infty]$,
{\em the Fenchel conjugate (a.k.a.~the Legendre transform) of $f$} is 
$f^*: \euclidspace \rightarrow (-\infty,\infty]:u\mapsto
\sup_{x\in\euclidspace} (\innerprod{x}{u} - f(x))$.
The conjugate $f^{**}:= (f^*)^*$ of $f^*$ is called 
the biconjugate of $f$.
In general, $f^*$ is l.s.c.~and convex
\cite[Proposition 13.13]{combettes}.\footnote{
It may happen that $\dom f^* = \varnothing$. For instance, the conjugate
of $f:\real\rightarrow\real:x\mapsto -x^2$ is 
$f^*(u)=+\infty$ for every $u\in\real$.
}
If $\dom f^* \neq \varnothing$ (i.e., $f^*$ is proper), or equivalently
if $f$ possesses a continuous {\em affine minorant}\footnote{
Function $f$ has a continuous affine minorant if
there exist some $(a,b)\in \euclidspace\times \real$  such that
$f(x) \geq \innerprod{a}{x}+b$ for every $x\in\euclidspace$.}, 
then $f^{**} = \breve{f}$;
otherwise $f^{**}(x)= -\infty$ for every
$x\in\euclidspace$
\cite[Proposition 13.45]{combettes}. 
Here, $\breve{f}$ is {\em the l.s.c.~convex envelope} of
 $f$.\footnote{
The l.s.c.~convex envelope $\breve{f}$ is the
largest l.s.c.~convex function $g$ such that 
$f(x) \geq g(x)$, $\forall x\in\euclidspace$.}

\begin{fact}
[Fenchel--Moreau Theorem \cite{combettes}]
\label{fact:fenchel_moreau}
Given a proper function $f:\euclidspace\rightarrow (-\infty,+\infty]$,
the following equivalence and implication hold:
$f\in\Gamma_0(\euclidspace) \Leftrightarrow f=f^{**}
\Rightarrow f^*\in\Gamma_0(\euclidspace)$.
\end{fact}

Let $f:\euclidspace\rightarrow (-\infty,+\infty]$ be proper.
Then, the set-valued operator $\partial f:\euclidspace \rightarrow
2^\euclidspace$
such that
\begin{equation}
\partial f: x \mapsto
\{z\in\euclidspace\mid \innerprod{y\! -\! x}{z} + f(x) \leq f(y), ~\forall y\in\euclidspace\}
\label{eq:def_subdifferential}
\end{equation}
is the subdifferential of $f$ \cite{combettes}.

\begin{fact}
\label{fact:subdifferential}
Let $f:\euclidspace\rightarrow (-\infty,+\infty]$ be proper.
Then, the following statements hold.

\begin{enumerate}
 \item $\partial f$ is monotone {\rm \cite[Example 20.3]{combettes}}.
 \item $\partial f(x) \subset (\partial f^*)^{-1}(x)$ for every $x\in\euclidspace$
and 
$(\partial f)^{-1}(u) \subset \partial f^*(u)$ for every $u\in\euclidspace$
{\rm \cite[Proposition 16.10]{combettes}}.
\end{enumerate}

\end{fact}

\begin{fact}
\label{fact:subdifferential_convex}
Let $f\in\Gamma_0(\euclidspace)$. 
Then, the following statements hold.

\begin{enumerate}
 \item $\partial f$ is maximally monotone {\rm \cite[Theorem 20.25]{combettes}}.
 \item $(\partial f)^{-1} = \partial f^*$ {\rm \cite[Corollary
       16.30]{combettes}}.
\item $\partial f(x)\neq \varnothing$ if $f$ is continuous at
 $x\in\euclidspace$
{\rm \cite[Proposition 16.17]{combettes}}.
\item
$\partial f(x)= \{\nabla f(x)\}$
 if $f$ is (G\^ateaux) differentiable
 with its (G\^ateaux) derivative $\nabla f$
{\rm \cite[Proposition 17.31]{combettes}}.
\end{enumerate}
 
\end{fact}

%


\subsection{Proximity Operator of Nonconvex Function}
\label{subsec:prox}

\begin{definition}[Proximity operator \cite{rockafellar10,bauschke21}]
Let $f:\euclidspace\rightarrow (-\infty,+\infty]$ be proper.
The proximity operator of $f$
of index $\gamma \in\real_{++}$ is then defined by
\begin{equation}
\hspace*{-.4em}
{\prox}_{\gamma f}:
\euclidspace \! \rightarrow \! 2^\euclidspace:
x \mapsto \argmin_{y \in \mathcal{H}}
\!\Big(f (y) + \frac{1}{2 \gamma}
\norm{x \!- \!y}^2\Big),
\label{eq:def_prox}
\end{equation}
which is {\em set-valued} in general. 
\end{definition}
We present a slight extension of the previous result \cite[Lemma
1]{yukawa_molgrad24} below.



\begin{lemma}
\label{lemma:prox_decomp}
Let $f:\euclidspace\rightarrow (-\infty,+\infty]$ be a proper function.
Then, given every positive constant $\gamma\in\real_{++}$,
it holds that
\begin{equation}
\prox_{\gamma f} = \left[
\partial \Big(f + \frac{1}{2\gamma}\norm{\cdot}^2\Big)
\right]^{-1}\circ (\gamma^{-1}\id),
\end{equation}
which is monotone.
\end{lemma}
\begin{proof}
The case of $\gamma:=1$ is given in
\cite[Lemma 1]{yukawa_molgrad24},
from which the following equivalence can be verified:
\begin{align}
\hspace*{0em} &~ p\in\prox_{\gamma f} (x)
= 
\left[
 \partial 
  \left(
    \gamma f + (1/2)\norm{\cdot}^2
 \right)
\right]^{-1}(x)
\nonumber \\
\hspace*{-0em} \Leftrightarrow  &~ 
x\!\in 
 \partial \!
  \left(\!
    \gamma f + (1/2)\norm{\cdot}^2
 \right)\!
(p)
= 
    \gamma \partial \!
  \left(
 f + (\gamma^{-1}/2) \norm{\cdot}^2
 \right)\!
(p)
\nonumber \\
%
\hspace*{-1em} \Leftrightarrow  &~ 
p   \in 
\left[ \partial 
  \left(
 f +  (\gamma^{-1}/2) \norm{\cdot}^2
 \right)
\right]^{-1}
(\gamma^{-1}x).
\end{align}
Finally, monotonicity of $\prox_{\gamma f}$ can readily be verified by
combining Facts \ref{fact:monotone} and \ref{fact:subdifferential}.1.
\end{proof}

\begin{definition}[Single-valued proximity operator \cite{yukawa_molgrad24}]
If ${\prox}_{\gamma f}$ is single-valued, 
it is denoted by ${\sprox}_{\gamma f}:\euclidspace\rightarrow \euclidspace$,
which is referred to as the s-prox operator
of $f$ of index $\gamma$.
\end{definition}

As a particular instance,
if $f+(\eta/2)\norm{\cdot}^2\in\gzh$ for some constant $\eta\in (-\infty, \gamma^{-1})$,
existence and uniqueness of minimizer is automatically
ensured.
Here, when $\eta\leq 0$, $f$ is convex (strongly convex, more
specifically, when $\eta<0$), and thus
${\sprox}_{\gamma f}$ reduces to the classical Moreau's proximity operator
\cite{moreau65}.

\begin{fact}[\cite{yukawa_molgrad24} MoL-Grad Denoiser]
\label{fact:weaklyconvex_necsuffcondition}

 Let $T:\euclidspace\rightarrow \euclidspace$.
Then, for every $\eta\in[0,1)$, the following two conditions are
 equivalent.\footnote{
The case of $\eta:=0$ is due to \cite{moreau65}.
}
\begin{enumerate}
 \item[(C1)] $T=\sprox_\varphi$
       for some $\varphi:\euclidspace\rightarrow (-\infty,+\infty]$
such that $\varphi + \left(\eta/2\right)\norm{\cdot}^2\in\Gamma_0(\euclidspace)$.
 \item[(C2)] $T$ is a $(1-\eta)^{-1}$-Lipschitz continuous gradient operator of
a  (Fr\'echet) differentiable convex function $\psi\in \gzh$.
In other words,
$T$ is a MoL-Grad (monotone Lipschitz-continuous gradient) denoiser.
\end{enumerate}
If {\rm (C1)}, or equivalently {\rm (C2)}, is satisfied, 
then it holds that $\varphi = \psi^* - (1/2)\norm{\cdot}^2$
$(\Leftrightarrow \psi = (\varphi + (1/2)\norm{\cdot}^2)^*)$.

\end{fact}


\section{Proximal Inclusion and Conversion}
\label{sec:general_case}

Suppose that a given discontinuous (monotone) operator $T$
is a selection of the (set-valued) proximity operator
of some proper function.
%
Then, there exists a principled way of constructing
a continuous relaxation of $T$
which is the s-prox operator of a certain weakly convex function.
This is the main claim of this article 
that will be supported by the key results ---
{\em proximal inclusion and conversion}.
Some other results on set-valued proximity operators are also presented.
All results (lemma, propositions, theorems, corollaries) in what
follows are the original contributions of this work.

\subsection{Interplay of Maximality of Monotone Operator and
Weak Convexity of Function}
\label{subsec:interplay_max_cyclic_monotone_weak_convexity}

Fact \ref{fact:weaklyconvex_necsuffcondition} presented above
gives an interplay between $\eta$-weakly convex functions for
$\eta\in[0,1)$ and the $(1-\eta)^{-1}$-Lipschitz continuous
gradients of smooth convex functions, stemming essentially from 
the duality between strongly convex functions and smooth convex functions.
Now, the question is the following: is there any such relation in the case of $\eta \in [1,+\infty)$?

While the proximity operator $\sprox_{\varphi}$ is (Lipschitz) continuous in the case of $\eta<1$,
the case of $\eta\geq 1$ (more specifically, the case in which
$\varphi+ (\eta/2)\norm{\cdot}^2\not\in\Gamma_0(\euclidspace)$ for any $\eta<1$)
includes those functions of which the ``set-valued'' proximity operator
contains a discontinuous shrinkage operator as its selection.
The following proposition concerns the case of $\eta=1$,
which will be linked to the case of $\eta>1$ later in Section 
\ref{sec:general_case}-\ref{subsec:prox_inclusion}.

\begin{proposition}
\label{proposition:max_cyc_mono_weak_convexity}

Given a set-valued operator $\mathsf{T}:\euclidspace\rightarrow 2^\euclidspace$,
the following statements are equivalent.
\begin{enumerate}
 \item $\mathsf{T}=\partial \psi$ for some convex function $\psi\in\Gamma_0(\euclidspace)$.
 \item $\mathsf{T}=\prox_{\phi}$ 
       for some $\phi+(1/2)\norm{\cdot}^2\in\Gamma_0(\euclidspace)$.
\end{enumerate}
Moreover, if statements {\rm 1 -- 2} are true, it holds that
$\phi =  \psi^* - (1/2)\norm{\cdot}^2$
$(\Leftrightarrow \psi = (\phi + (1/2)\norm{\cdot}^2)^*)$.

\end{proposition}
\begin{proof}
The equivalence 1) $\Leftrightarrow$ 2)
can be verified by showing the following equivalence basically:
\begin{align}
 &~ \mathsf{T}= \partial \psi, ~ \exists
 \psi\in\Gamma_0(\euclidspace) \nonumber\\
\Leftrightarrow &~ \mathsf{T}= 
(\partial \psi^*)^{-1}=
\Big[
\partial \Big(\underbrace{\Big(\psi^* - \frac{1}{2}\norm{\cdot}^2\Big)}_{=:\phi} +
 \frac{1}{2}\norm{\cdot}^2\Big)
\Big]^{-1}
 \nonumber \\
 &\hspace*{1.1em} =\prox_{\psi^* - (1/2)\norm{\cdot}^2},
~~~ 
\exists
 \psi\in\Gamma_0(\euclidspace).
\nonumber
\end{align}
(Proof of $\Rightarrow$)
The last equality can be verified by
Lemma \ref{lemma:prox_decomp},
and $\phi+ (1/2)\norm{\cdot}^2
= \psi^* \in\Gamma_0(\euclidspace)$ follows
by Fact \ref{fact:fenchel_moreau}.

\noindent 
(Proof of $\Leftarrow$)
Letting $\psi:=  (\phi+(1/2)\norm{\cdot}^2)^*$,
we have $\psi\in\gzh$
again by Fact \ref{fact:fenchel_moreau}
so that $\mathsf{T}=(\partial (\phi+(1/2)\norm{\cdot}^2))^{-1} = 
\partial [(\phi+(1/2)\norm{\cdot}^2)^*] = \partial \psi $.
\end{proof}

Proposition \ref{proposition:max_cyc_mono_weak_convexity}
bridges the subdifferential of
convex function and the proximity operator of 1-weakly convex function,
indicating an interplay between 
maximality of monotone operators and 1-weak convexity of functions.

\begin{remark}[Role of Monotonicity]
\label{remark:max_monotone_weak_convexity}
From
Proposition \ref{proposition:max_cyc_mono_weak_convexity} together with
Fact \ref{fact:subdifferential_convex}.1,
$\phi +(1/2)\norm{\cdot}^2\in \Gamma_0(\euclidspace)$
implies that
$\mathsf{T}=\prox_{\phi}$ is maximally monotone.
Viewing the proposition in light of Rockafeller's cyclic monotonicity theorem
{\rm \cite[Theorem 22.18]{combettes}}, moreover,
one can see that
$\mathsf{T}=\prox_{\phi}$ with $\phi +(1/2)\norm{\cdot}^2\in \Gamma_0(\euclidspace)$
if and only if $\mathsf{T}$ is 
maximally cyclically monotone.\footnote{
An operator $\mathsf{T}:\euclidspace\rightarrow 2^\euclidspace$ is cyclically monotone if,
for every integer $n\geq 2$, $u_i\in \mathsf{T}(x_i)$ for $i\in\{1,2,\cdots,n\}$
and $x_{n+1}=x_1$ imply $\sum_{i=1}^{n} \innerprod{x_{i+1}-x_i}{u_i}\leq
 0$.
It is {\em maximally cyclically monotone} if no other cyclically monotone operator
has its graph containing gra $\mathsf{T}$ properly \cite{combettes}.
An operator $\mathsf{T}:\euclidspace\rightarrow 2^\euclidspace$
is maximally cyclically monotone 
if and only if
$\mathsf{T}=\partial f$ for some $f\in \Gamma_0(\euclidspace)$
(Rockafeller's cyclic monotonicity theorem).
It is clear under Fact \ref{fact:subdifferential_convex}.1 that
a maximally cyclically monotone operator is maximally monotone;
the converse is 
(true when $\euclidspace=\real$ but) not true in general.
}
In words, maximal cyclic monotonicity characterizes
the property that $T$ can be expressed as the proximity operator 
(which is set-valued in general) of a 1-weakly convex function.
Note that, whereas
the proximity operator of a 1-weakly convex function
is maximally monotone, the converse is not true in general;
i.e., ``maximal monotonicity'' itself does not ensure
that $T$ can be expressed as the proximity operator of 
a 1-weakly convex function.

In Fact \ref{fact:weaklyconvex_necsuffcondition}, 
monotonicity plays a role of ensuring the convexity of $\psi$.
Indeed, the assumption on $T$ in Fact \ref{fact:weaklyconvex_necsuffcondition}
(i.e., the condition for MoL-Grad denoiser) implies
maximal cyclic monotonicity, 
since $\eta$-weak convexity for $\eta:=1-\beta\in(0,1)$ implies 1-weak convexity.
The assumption required for $\eta\in(0,1)$ is actually even stronger than 
maximal cyclic monotonicity (which is required for $\eta=1$).
\end{remark}

\subsection{Proximal Inclusion}
\label{subsec:prox_inclusion}

We start with the definition of
the l.s.c.~1-weakly-convex envelope.
\begin{definition}[L.s.c.~1-weakly-convex envelope]
Let $f:\euclidspace\rightarrow (-\infty,+\infty]$ be a proper function
such that $(f+(1/2)\norm{\cdot}^2)^*$ is proper as well.
Then, $(f+(1/2)\norm{\cdot}^2)^{**} - (1/2)\norm{\cdot}^2$
is the l.s.c.~1-weakly-convex envelope of $f$.
The notation $\widetilde{(\cdot)}$ will be used 
to denote the l.s.c.~1-weakly-convex envelope, such as
$\widetilde{f} :=(f+(1/2)\norm{\cdot}^2)^{**} - (1/2)\norm{\cdot}^2$.
The envelope $\widetilde{f}$ is the largest (proper) l.s.c.
 1-weakly-convex function $\phi:\euclidspace\rightarrow  (-\infty,+\infty]$
 such that $f(x) \geq \phi(x)$, $\forall x\in\euclidspace$.
\end{definition}

The first key result is presented below.

\begin{theorem}[Proximal inclusion]
\label{theorem:prox_inclusion} 

Let $f:\euclidspace\rightarrow (-\infty,+\infty]$ be a proper function
such that $(f+(1/2)\norm{\cdot}^2)^*$ is proper as well.
Then, the following inclusion holds
between the proximity operators of $f$ and its l.s.c.~1-weakly-convex
 envelope $\widetilde{f}$:
\begin{equation}
 \prox_f(x)\subset \prox_{\widetilde{f}}(x), ~\forall x\in\euclidspace,
\end{equation}
i.e., ${\rm gra~} \prox_f \subset {\rm gra~} \prox_{\widetilde{f}}$.
\end{theorem}
\begin{proof}
By the assumption, we have
 $\big(f+(1/2)\norm{\cdot}^2\big)^*\in\Gamma_0(\euclidspace)$,
and thus  
$\widetilde{f} + (1/2)\norm{\cdot}^2 =
\big(f+(1/2)\norm{\cdot}^2\big)^{**}\in\Gamma_0(\euclidspace)$
by Fact \ref{fact:fenchel_moreau}.
Hence, for every $u\in\euclidspace$, it follows that
\begin{align}
 &\hspace*{0em} \prox_f(x)= [\partial (f+(1/2)\norm{\cdot}^2)]^{-1}(x) \nonumber\\
& \hspace*{-.2em}\subset \partial [(f+(1/2)\norm{\cdot}^2)^*](x)
=  [\partial ((f+(1/2)\norm{\cdot}^2)^{**})]^{-1}(x)\nonumber\\
& \hspace*{.8em}=  [\partial
 (\widetilde{f}+(1/2)\norm{\cdot}^2)]^{-1}(x)= \prox_{\widetilde{f}}(x),
\end{align}
where the first and last equalities are due to Lemma
 \ref{lemma:prox_decomp},
the inclusion is due to Fact \ref{fact:subdifferential}.2,
and the third equality is due to 
Fact \ref{fact:subdifferential_convex}.2.
\end{proof}

Theorem \ref{theorem:prox_inclusion} implies that,
if a discontinuous operator is a selection of
$\prox_f$ for a nonconvex function $f$,
it can also be expressed as a selection of
$\prox_{\widetilde{f}}$ for a 1-weakly-convex function $\widetilde{f}$,
which actually coincides with the l.s.c.~1-weakly-convex envelope of
$f$.
%
This fact indicates that, when seeking for a shrinkage operator as (a selection
       of) the proximity operator, one may restrict attention to 
 the class of weakly convex functions.
%

%
%
%

We remark that Theorem \ref{theorem:prox_inclusion} is trivial if
$f+(1/2)\norm{\cdot}^2\in \Gamma_0(\euclidspace)$,
because $\widetilde{f}=f$ in that case by Fact \ref{fact:fenchel_moreau}.
Hence, our primary focus in Theorem \ref{theorem:prox_inclusion} is on
$\eta$-weakly convex functions for $\eta>1$,
although Theorem \ref{theorem:prox_inclusion} itself has no such a
restriction.
The following corollary is a direct consequence of 
Fact \ref{fact:weaklyconvex_necsuffcondition} ($\eta<1$),
Proposition
\ref{proposition:max_cyc_mono_weak_convexity} ($\eta=1$),
and Theorem \ref{theorem:prox_inclusion} 
($\eta>1$).

\begin{corollary}
\label{corollary:prox_weak_convexity}
Let $f+(\eta/2)\norm{\cdot}^2\in\Gamma_0(\euclidspace)$
for $\eta\in(-\infty,+\infty]$.
Then, the following statements hold.
\begin{enumerate}
 \item Assume that $\eta \leq 1$.
Then, $\prox_{f}$ is maximally cyclically monotone
(see Remark \ref{remark:max_monotone_weak_convexity}).
In particular, if $\eta <1$,
the proximity operator is single valued, and
$\sprox_f$ is $(1-\eta)^{-1}$-Lipschitz continuous.

 \item 
Assume that $\eta >1$;
more specifically, assume that
       $f+(1/2)\norm{\cdot}^2\not\in\Gamma_0(\euclidspace)$).
Assume also that (i) $(f+(1/2)\norm{\cdot}^2)^*$ is proper and that
(ii) there exists $\hat{x}\in \dom \partial
       (f+\frac{1}{2}\norm{\cdot}^2)^{**}$
such that $(f+\frac{1}{2}\norm{\cdot}^2)^{**}(\hat{x}) \neq 
f(\hat{x})+\frac{1}{2}\norm{\hat{x}}^2$.
Then,
$ \prox_f(\hat{x})\subsetneq \prox_{\widetilde{f}}(\hat{x})$, which implies that
$\prox_{f}$ cannot be maximally monotone.\footnote{
Assumption (ii) might be able to be removed.
}
\end{enumerate}

\end{corollary}

\begin{remark}[On Theorem \ref{theorem:prox_inclusion}]
\label{remark:prox_inclusion}
\begin{enumerate}
 \item 
If the properness assumption of 
$(f+(1/2)\norm{\cdot}^2)^*$ is violated, it holds that
$\varnothing = \dom (f+(1/2)\norm{\cdot}^2)^*
\supset \dom \partial (f+(1/2)\norm{\cdot}^2)^*$
{\rm \cite[Proposition 16.4]{combettes}},
which implies that $\dom \partial (f+(1/2)\norm{\cdot}^2)^* = \varnothing$.
Thus, by Lemma \ref{lemma:prox_decomp}, 
it can be verified,
for every $x\in\euclidspace$,
that
$\varnothing = \partial (f+(1/2)\norm{\cdot}^2)^*(x)
\supset \partial (f+(1/2)\norm{\cdot}^2)^{-1}(x)
=\prox_f(x)$
 (cf.~Section \ref{sec:preliminaries}),
which implies that
$\prox_f(x)=\varnothing$.

 \item By Fact \ref{fact:fenchel_moreau},
$\widetilde{f}=f$ if and only if
$f+(1/2)\norm{\cdot}^2\in\Gamma_0(\euclidspace)$.

\item The operator
$\prox_{\widetilde{f}}$ 
is maximally cyclically monotone.

\end{enumerate}
\end{remark}


\subsection{Proximal Conversion}
\label{subsec:conversion}

The second key result presented below gives a principled way of 
converting the set-valued proximity operator of 1-weakly convex
function to a MoL-Grad denoiser.

\begin{theorem}[Proximal conversion]
\label{theorem:prox_conversion}
Let $\phi$ be a function such that
$\phi+(1/2)\norm{\cdot}^2\in\Gamma_0(\euclidspace)$.
Then, the proximity operator of the
$(\delta+1)^{-1}$-weakly convex function $\phi/(\delta+1)$
for the relaxation parameter $\delta\in\real_{++}$ can be expressed as
\begin{equation}
\sprox_{\phi/(\delta+1)} = 
 \left[
\prox_{\phi}^{-1}
+\delta \id
\right]^{-1}\circ 
(\delta+1)\id,
\label{eq:relation_hard_firm}
\end{equation}
which is $(1+1/\delta)$-Lipschitz continuous.
\end{theorem}
%
\begin{proof}
Since $\phi$ is 1-weakly convex, it is clear that
$\phi/(\delta+1)$ is $\eta$-weakly convex for
$\eta:= (\delta+1)^{-1}\in(0,1)$.
Letting $f:=\phi$ and $\gamma:=1/(\delta+1)$ 
in Lemma \ref{lemma:prox_decomp} yields
\begin{equation}
 \sprox_{\phi/(\delta+1)} = \left[
\partial (\phi + ((\delta+1)/2)\norm{\cdot}^2)
\right]^{-1}\circ (\delta+1)\id.
\end{equation}
Here, since $\phi+(1/2)\norm{\cdot}^2\in\Gamma_0(\euclidspace)$,
it holds that
$\partial (\phi + ((\delta+1)/2)\norm{\cdot}^2)
= 
\partial (\phi + (1/2)\norm{\cdot}^2)
+ \partial [(\delta/2)\norm{\cdot}^2]
= 
\partial (\phi + (1/2)\norm{\cdot}^2)
+ \delta \id
=  \prox_{\phi}^{-1} + \delta \id$, 
where the last equality can be verified
by letting $f:=\phi$ and $\gamma:=1$ in
Lemma \ref{lemma:prox_decomp}.
From Fact \ref{fact:weaklyconvex_necsuffcondition},
the Lipschitz constant is given by
$(1-\eta)^{-1}=1+1/\delta$,
which completes the proof.
\end{proof}

\begin{remark}
\label{remark:continuous_relaxation}

Theorems
\ref{theorem:prox_inclusion}
and \ref{theorem:prox_conversion}
give a constructive proof for 
the existence of continuous relaxation of discontinuous operator, say $T$,
provided that $T$ is a selection of $\prox_{f}$ for some
proper function  $f:\euclidspace\rightarrow (-\infty,+\infty]$.
Such an $f$ exists if and only if there exists 
a proper function $\psi:\euclidspace\rightarrow (-\infty,+\infty]$
such that $T^{-1}(u)\subset \partial \psi(u)$, $\forall u\in \range T$.
Concrete examples will be given in Section \ref{sec:application}.
\end{remark}

\subsection{Surjectivity of Mapping and Continuity of Its Associated Function}
\label{subsec:surjectivity}

An interplay between
surjectivity of an operator $T$ and 
continuity (under weak convexity) of a function 
can be seen through a `lens' of proximity operator.\footnote{
It is known that convexity of $f:\euclidspace \rightarrow \real$
implies continuity of $f$ when $\euclidspace$ is finite dimensional 
\cite[Corollary 8.40]{combettes}.
}

\begin{lemma}
\label{lemma:domphi_convexity_continuity}
Let $f:\euclidspace\rightarrow (-\infty,+\infty]$ be 
a proper function.
Then, the following two statements are equivalent.
\begin{enumerate}
 \item $\dom \partial f = \euclidspace$.
 \item $f:\euclidspace\rightarrow \real$ is convex and continuous over
       $\euclidspace$.
\end{enumerate}
\end{lemma}
\begin{proof}
 1) $\Rightarrow$ 2): 
By 
\cite[Proposition 16.5]{combettes} and
Fact \ref{fact:fenchel_moreau},
it holds that
$\dom \partial f = \euclidspace 
\Rightarrow f=f^{**}
\Leftrightarrow f\in\Gamma_0(\euclidspace)$.
Hence, invoking
\cite[Propositions 16.4 and 16.27]{combettes},
we  obtain $\euclidspace = \dom \partial f \subset \dom f =
 \euclidspace
\Rightarrow {\rm int} ~ \dom f = \cont f = \euclidspace$.

\noindent 2) $\Rightarrow$ 1): Clear from \cite[Proposition 16.17(ii)]{combettes}.
\end{proof}

\begin{proposition}
\label{proposition:surjectivity_and_continuityweakconvexity} 

Let $f:\euclidspace\rightarrow (-\infty,+\infty]$
be  a proper l.s.c.~function.
Define an operator $U:\euclidspace\rightarrow \euclidspace$ 
such that
$U(x)\in \prox_f (x)$ for every $x\in\euclidspace$.
Consider the following two statements.
\begin{enumerate}
 \item  $\range U = \euclidspace$. 

 \item $f + (1/2)\norm{\cdot}^2$ is convex and continuous.
\end{enumerate}
Then, 1) $\Rightarrow$ 2).
Assume that $U=  \sprox_{f}$; i.e.,
$f + (1/2)\norm{\cdot-x}^2$ has a unique
 minimizer for every $x\in\euclidspace$.
Then, 1) $\Leftrightarrow$ 2).

\end{proposition}

\begin{proof}
 By the assumption, it holds that
$\euclidspace 
= \range U
\subset \range [\partial (f+(1/2)\norm{\cdot}^2)]^{-1} 
= \dom \partial (f+(1/2)\norm{\cdot}^2)
= \euclidspace$.
Hence, by Lemma \ref{lemma:domphi_convexity_continuity},
$\range U=\euclidspace$ implies convexity and continuity of $f+(1/2)\norm{\cdot}^2$.
Under the assumption that $U=  \sprox_{f}$, 
by Lemma \ref{lemma:domphi_convexity_continuity} again, 
$\range U= \dom \partial (f+(1/2)\norm{\cdot}^2) = \euclidspace$.
\end{proof}


\section{Application: Continuous Relaxation of Discontinuous Operator}
\label{sec:application}

As an illustrative example, we first show how hard shrinkage
is converted to a continuous operator by leveraging Theorems
\ref{theorem:prox_inclusion} and \ref{theorem:prox_conversion}.
We then apply the same idea to the ROWL-based discontinuous operator to
obtain its continuous relaxation.

\subsection{An Illustrative Example: Converting Discontinuous Hard
  Shrinkage to Continuous Firm Shrinkage}
\label{subsec:l0_mc}

The hard shrinkage operator with the threshold $\tau\in\real_{++}$
is defined by \cite{blumensath08}
 \begin{equation}
\hspace*{-.5em}
  {\rm hard}_{\tau}:
\real \!\rightarrow \real:
 x\mapsto x 1(\abs{x}> \tau):=
 \begin{cases}
 0, & \!\! \mbox{if } \abs{x} \leq \tau, \\
 x, &\!\! \mbox{if } \abs{x} > \tau,
 \end{cases}
 \end{equation}
for which it holds that
\begin{align}
\hspace*{-1em}{\rm hard}_{\tau}(x)\in
\prox_{g_\tau} (x) = \begin{cases}
 \{0\}, & \mbox{if } \abs{x} < \tau,  \\
 \{0,x\},  & \mbox{if } \abs{x}= \tau,  \\
\{x\}, & \mbox{if } \abs{x}> \tau,
			\end{cases}\label{eq:prox_g}
\end{align}
where
\begin{equation}
 g_\tau(x):= \frac{\tau^2}{2}\norm{x}_0=
 \begin{cases}
 0, & \mbox{if } x=0, \\
\dfrac{\tau^2}{2}, & \mbox{if } x\neq 0,
 \end{cases}~~~\forall x\in \real.
\end{equation}
%
%
The l.s.c.~1-weakly convex envelope
$\widetilde{g}_\tau:=  
(g_\tau +(1/2)(\cdot)^2)^{**} - (1/2)(\cdot)^2$ of $g_\tau$
is given by
(see Fig.~\ref{fig:functions})
\begin{equation}
\hspace*{-.7em}
\widetilde{g}_\tau(x) \!=\!
\tau \varphi_{{\tau}}^{\rm MC}(x) =
 \begin{cases}
 \tau\abs{x} - \dfrac{1}{2} x^2, & 
 \!  \!\!\mbox{if }\! \abs{x}\leq \tau,  \\
  \dfrac{\tau^2}{2}, & 
  \! \!\!\mbox{if }\! \abs{x}> \tau.
		\end{cases}
\label{eq:gtilde}
\end{equation}
Here,
$\varphi_{\tau_2}^{\rm MC}(x):= 
\begin{cases}
\abs{x} - \frac{1}{2\tau_2}x^2, & \mbox{if } \abs{x}\leq \tau_2, \\
\frac{1}{2}\tau_2, & \mbox{if } \abs{x}> \tau_2,
\end{cases}
$
is the MC penalty
of the parameter $\tau_2\in\real_{++}$
\cite{zhang,selesnick}.

 \begin{figure}[t!]
  \psfrag{L0}[Bl][Bl][.7]{$\norm{\cdot}_0$}
  \psfrag{L0+}[Bl][Bl][.7]{$\norm{\cdot}_0+(1/2)(\cdot)^2$}
  \psfrag{MC}[Bl][Bl][.7]{$\widetilde{\norm{\cdot}}_0$}
  \psfrag{MC rho}[Bl][Bl][.7]{$\widetilde{\norm{\cdot}}_0/(\delta+1)$}
  \psfrag{biconjugation}[Bl][Bl][.7]{$(\norm{\cdot}_0 + (1/2)(\cdot)^2)^{**}$}
  \psfrag{x}[Bl][Bl][1]{$x$}

 \centering
\includegraphics[width=8cm]{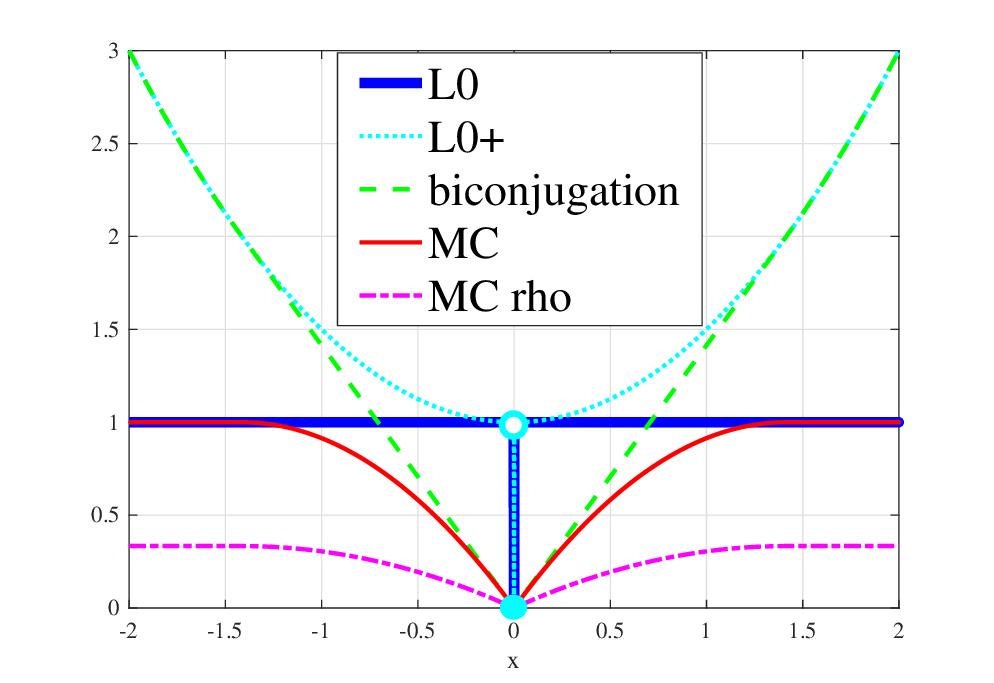}

 \caption{$\ell_0$ pseudo norm
$\norm{\cdot}_0$,
 its l.s.c.~1-weakly-convex envelope
$\widetilde{\norm{\cdot}}_0$, and related functions.
The biconjugate
$(\norm{\cdot}_0 + \norm{\cdot}^2/2)^{**}$ is 
the l.s.c.~convex envelope of
$\norm{\cdot}_0 + \norm{\cdot}^2/2$.
}
 \label{fig:functions}
 \end{figure}

The proximity operator of $\widetilde{g}_\tau$ is given
by
\begin{align}
\prox_{\widetilde{g}_\tau}:&~ x\mapsto
 \begin{cases}
  \{0\}, & \mbox{if } \abs{x} < \tau,  \\
  \overline{\rm conv} \{0,x\},  & \mbox{if } \abs{x}= \tau,
  \\
  \{x\}, & \mbox{if } \abs{x}> \tau,
\end{cases}
\label{eq:prox_gtilde}
\end{align}
where $\overline{\rm conv}$ denotes the closed convex hull
($\overline{\rm conv}\{0,\tau\}= [0,\tau]$ and
$\overline{\rm conv}\{0,-\tau\}= [-\tau,0]$).
Comparing \eqref{eq:prox_g} and \eqref{eq:prox_gtilde},
it can be seen that
\begin{equation}
 ({\rm hard}_{\tau}(x)\in)
\prox_{g_\tau}(x) \subset
\prox_{\widetilde{g}_\tau}(x),~
\forall x\in \real,
\label{eq:hard_inclusion}
\end{equation}
as consistent with Theorem \ref{theorem:prox_inclusion}.
This implies that ${\rm hard}_{\tau}$ is also {\em a selection} of
the proximity operator
$\prox_{\widetilde{g}_\tau}$ of the
1-weakly convex function $\widetilde{g}_\tau$,
which is maximally monotone.
Indeed, $\prox_{\widetilde{g}_\tau}$ in
\eqref{eq:prox_gtilde} is
the (unique) maximally monotone extension of ${\rm hard}_{\tau}$
{\rm \cite{bayram15}} (cf.~Remark \ref{remark:prox_inclusion}).
%

We now invoke
Theorem \ref{theorem:prox_conversion}
to obtain the continuous operator
\begin{align}
\sprox_{\widetilde{g}_\tau/(\delta+1)} =&~
\Big[
\prox_{\widetilde{g}_\tau}^{-1} + \delta {\rm Id}
\Big]^{-1} \circ (\delta+1){\rm Id}\nonumber\\
=&~ {\rm firm}_{\tau/(\delta+1),\tau},
\label{eq:sprox_firm}
\end{align}
where
the firm shrinkage operator \cite{gao97}
for the thresholds $\tau_1\in \real_{++}$,
$\tau_2 \in (\tau_1,+\infty)$,
is defined by
${\rm firm}_{\tau_1,\tau_2}:=
\sprox_{\tau_1 \varphi_{\tau_2}^{\rm MC}}
:\real\rightarrow \real:
x \mapsto 0 ~
\mbox{if } \abs{x}\leq\tau_1;
x \mapsto {\rm sign}(x)\frac{\tau_2 (\abs{x} - \tau_1)}{\tau_2 -
 \tau_1} ~
\mbox{if }\tau_1 < \abs{x} \leq \tau_2;
x \mapsto x ~\mbox{if } \abs{x}>\tau_2$.
We remark that $\tau_1 \varphi_{\tau_2}^{\rm MC}$ is
$\tau_1/\tau_2$-weakly convex
($\widetilde{g}_\tau/(\delta+1)$ in \eqref{eq:sprox_firm} 
is $1/(\delta+1)$-weakly convex)
and its ``single-valued'' proximity operator gives firm
shrinkage, while $\widetilde{g}_\tau$ in \eqref{eq:hard_inclusion}
is 1-weakly convex and a selection of its ``set-valued'' proximity
operator gives hard shrinkage.
We also mention that the limit of the continuous operator 
${\rm firm}_{\tau/(\delta+1),\tau}$
with respect to the relaxation parameter $\delta$
coincides with the discontinuous operator ${\rm hard}_{\tau}$;
i.e., $ \lim_{\delta \downarrow 0} 
              {\rm firm}_{\tau/(\delta+1),\tau}(x) = 
\widehat{\rm hard}_{\tau}(x)
=x 1(\abs{x} \geq \tau) \in \prox_{g_{\tau}}(x)$
for every $x\in\real$
(see Fig.~\ref{fig:firm}).
On the other hand, the soft shrinkage operator with threshold
$\tau\in\real_{++}$ is characterized by
${\rm soft}_\tau = \lim_{\tau_2\rightarrow +\infty} {\rm firm}_{\tau,\tau_2}$.

 \begin{figure}[t!]
  \psfrag{x}[Bc][Bc][1]{$x$}
  \psfrag{delta}[Bc][Bc][1]{$\delta\downarrow 0$}
  \psfrag{delta1}[Br][Br][1]{$\delta =1$}
  \psfrag{delta3}[Br][Br][1]{$\delta =3$}
  \psfrag{delta13}[Br][Br][1]{$\delta =1/3$}

 \centering
\includegraphics[width=6cm]{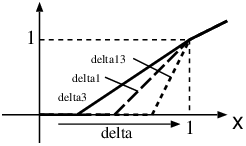}

 \caption{Pointwise convergence of ${\rm firm}_{1/(\delta+1),1}$ to
  $\widehat{\rm hard}_1$ as $\delta\downarrow 0$.}
 \label{fig:firm}
 \end{figure}

%

Fig.~\ref{fig:shrinkage}(a)--(g) illustrates the process of obtaining
a continuous relaxation $\sprox_{\frac{2}{3}\widetilde{g}_1}$
of the discontinuous operator hard$_1$, corresponding to the case of
$\tau:=1$ and $\delta:=1/2$.
Comparing
the graphs of the discontinuous operator hard$_1$,
$\prox_{g_1}$, and $\prox_{\widetilde{g}_1} (=:\mathsf{H})$, 
one can observe that 
(see Theorem \ref{theorem:prox_inclusion} for the second inclusion)
\begin{equation}
 {\rm gra}~{\rm hard}_1
\subset
{\rm gra}~\prox_{g_1} 
\subset 
{\rm gra}~ \prox_{\widetilde{g}_1}.
\end{equation}
The maximally monotone operator $\prox_{\widetilde{g}_1}$ 
(see Remark \ref{remark:max_monotone_weak_convexity}) is then converted
to $\sprox_{(2/3)\widetilde{g}_1}= {\rm firm}_{2\tau/3,\tau}$ 
in a step-by-step manner using \eqref{eq:sprox_firm} 
(which is based on Theorem \ref{theorem:prox_conversion}).
%
%
Inspecting the figure 
under Corollary \ref{corollary:prox_weak_convexity},
Figs.~\ref{fig:shrinkage}(b) and \ref{fig:shrinkage}(h)
correspond to 
Corollary \ref{corollary:prox_weak_convexity}.2 
(not maximally monotone)
for $\eta :=+\infty$
and $\eta:=4$, respectively,
and 
Figs.~\ref{fig:shrinkage}(c) and \ref{fig:shrinkage}(g)
correspond to 
Corollary \ref{corollary:prox_weak_convexity}.1
(maximally monotone)
for $\eta:=1$ and $\eta:=2/3$, respectively.

Letting $\eta:=1/(\delta +1)$, 
\eqref{eq:sprox_firm} for $\delta\in\real_{++}$ concerns 
the case of $\eta\in (0,1)$.
In the case of $\eta\in[1,+\infty)$, on the other hand,
the proximity operator
of the $\eta$-weakly convex function  $\eta\widetilde{g}_\tau$
is set-valued.
Actually, it is not difficult to verify that 
$\prox_{\eta \widetilde{g}_\tau} = \prox_{\eta g_\tau}$
for every $\eta\in(1,+\infty)$.
Note here that this is not true for $\eta:=1$, 
i.e.,
$\prox_{\widetilde{g}_\tau} \neq \prox_{g_\tau}$,
as can be seen from Figs.~\ref{fig:shrinkage}(b) and \ref{fig:shrinkage}(c).
See Fig.~\ref{fig:shrinkage}(h) for the case of $\eta:=4$.

 \begin{figure}[t!]
  \psfrag{delta}[Bl][Bl][.8]{$\delta\downarrow 0$}
  \psfrag{3/2}[Bc][Bc][.8]{$\dfrac{3}{2}$}
  \psfrag{2/3}[Bc][Bc][.8]{$\dfrac{2}{3}$}
  \psfrag{x}[Bl][Bl][1]{$x$}
  \psfrag{inf}[Bl][Bl][1]{$\infty$}
  \psfrag{a}[Bl][Bl][1]{(i)}
  \psfrag{bd}[Bl][Bl][1]{(ii)}
  \psfrag{f}[Bl][Bl][1]{(vi)}

\centering
\begin{tabular}{cc}
\subfigure[hard$_1$]{
\includegraphics[width=3.6cm]{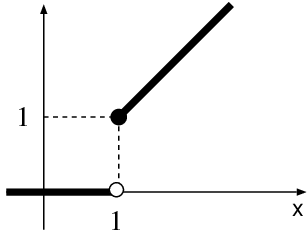}
}
 & 
\subfigure[$\prox_{g_1}$ ($\eta=+\infty$)]{
\includegraphics[width=3.6cm]{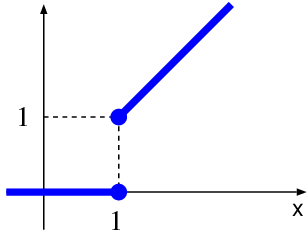}
}
\end{tabular}

\begin{tabular}{cc}
\subfigure[$\mathsf{H}:= \prox_{\widetilde{g}_1}$
($\eta=1$)]{
\includegraphics[width=3.6cm]{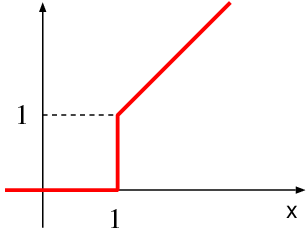}
}
 & 
\subfigure[$\mathsf{H}^{-1}$]{
\includegraphics[width=3.6cm]{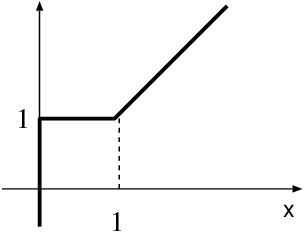}
}
\end{tabular}

\begin{tabular}{cc}
\subfigure[$\mathsf{H}^{-1}+0.5 {\rm Id}$]{
\includegraphics[width=3.6cm]{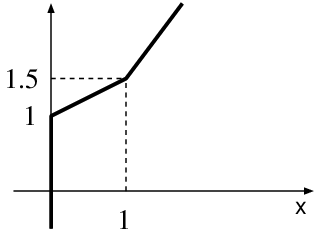}
}
 & 
\subfigure[$(\mathsf{H}^{-1}+ 0.5 {\rm Id})^{-1}$]{
\includegraphics[width=3.6cm]{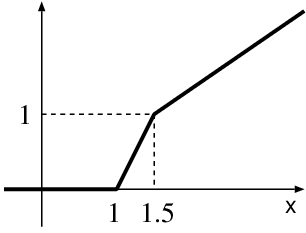}
}
\end{tabular}

\begin{tabular}{cc}
\subfigure[$\sprox_{(2/3)\widetilde{g}_1}$
($\eta=2/3$)]{
\includegraphics[width=3.6cm]{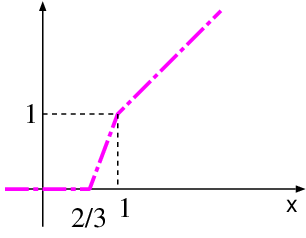}
}
 & 
\subfigure[$\prox_{4\widetilde{g}_1} = \prox_{4g_1}$
($\eta=4$)]{
\includegraphics[width=3.6cm]{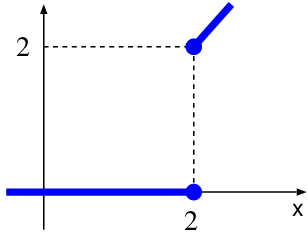}
}
\end{tabular}


  \caption{Graphs of 
(a) hard$_1$,
the proximity operators of
(b) $g_1=(1/2)\norm{\cdot}_0$,
(c) $\widetilde{g}_1$, and
(g) $(2/3)\widetilde{g}_1$,
(d)--(f) the intermediate operators
in conversion from (c) to (g), and
(h) the proximity operator of
$4\widetilde{g}_1$.
}
 \label{fig:shrinkage}
 \end{figure}


\subsection{eROWL Shrinkage:
Continuous Relaxation of 
ROWL Shrinkage Operator}
\label{subsec:ordered_weighted_L1}

We have seen that 
the discontinuous hard shrinkage operator is converted to
the continuous firm shrinkage operator via the transformation
from (i) to (vi) in Fig.~\ref{fig:shrinkage}(a).
We mimic this procedure for another discontinuous operator.

We consider the Euclidean case $\euclidspace:=\real^N$
for $N\geq 2$.
Let $\signal{w}\in\real_+^N$ be the weight vector
such that $0\leq w_1\leq w_2\leq \cdots \leq w_N$.
Given $\signal{x}\in\real^N$, we define
$\abs{\signal{x}}\in \real_+^N$ of which the $i$th component is given by
$\abs{x_i}$.
Let $\abs{\signal{x}}_{\downarrow} \in \real_{+
\downarrow}^N=\{\signal{x}\in \real_+^N\mid x_1\geq x_2\geq \cdots \geq x_N\}$ denote
a sorted version of $\abs{\signal{x}}$ in the nonincreasing order;
i.e., 
$\big[\abs{\signal{x}}_{\downarrow}\big]_1\geq
\big[\abs{\signal{x}}_{\downarrow}\big]_2\geq 
\cdots\geq 
\big[\abs{\signal{x}}_{\downarrow}\big]_N$.
The reversely ordered weighted $\ell_1$ (ROWL) 
penalty \cite{sasaki24} is defined by
$\Omega_{\signal{w}}(\signal{x}) := \signal{w}^\top
\abs{\signal{x}}_{\downarrow}$.
The penalty $\Omega_{\signal{w}}$ is nonconvex and thus not a norm.\footnote{If $w_1\geq w_2\geq \cdots \geq w_N$,
the function $\signal{x}\mapsto \signal{w}^\top
\abs{\signal{x}}_{\downarrow}$ 
is convex, and it is called
the ordered weighted $\ell_1$ (OWL) norm \cite{zeng15}.
}

In this case, the associated proximity operator will
be discontinuous. 
This implies in light of Fact \ref{fact:weaklyconvex_necsuffcondition} that 
$\Omega_{\signal{w}}$ is not even {\em weakly convex}.
To see this, let us consider the case of $N=2$, and let $(0\leq )w_1< w_2$.
In this case, 
by symmetry, 
the proximity operator is given by
$ \prox_{\Omega_{\signal{w}}} (\signal{x}) = 
{\rm sgn}(\signal{x}) \odot \prox_{\Omega_{\signal{w}}}
(\abs{\signal{x}}),
~\signal{x}\in\real^2$,
where ${\rm sgn}(\cdot)$ is the componentwise signum function,
$\odot$ denotes the Hadamard (componentwise) product,
and for $\signal{x}\in\real_+^2$
\begin{equation}
\hspace*{-.3em} \prox_{\Omega_{\signal{w}}} (\signal{x}) \!=\!
\begin{cases}
\!\{ (\signal{x} -
 \signal{w} )_+\}, & 
\hspace*{-.5em}\mbox{if } \!x_1 \!>\! x_2,  \\
\!\{ (\signal{x} - \signal{w} )_+ ,(\signal{x} - \signal{w}_{\downarrow} )_+ \}, & 
\hspace*{-.5em}\mbox{if } \!x_1 \!=\! x_2,  \\
\!\{ (\signal{x} - \signal{w}_{\downarrow} )_+\}, & 
\hspace*{-.5em} \mbox{if } \!x_1 \!<\! x_2.  \\
\end{cases}
\label{eq:prox_rowl}
\end{equation}
Here,
$\signal{w}_{\downarrow}:=[w_2, w_1]^\top$,
and $(\cdot)_+:\real^2\rightarrow \real^2:[y_1,y_2]^\top
\mapsto [\max\{y_1,0\}, \max\{y_2,0\}]^\top$ is the `ramp' function.
A selection
$R_{\rm ROWL}:\real^2\rightarrow \real^2:
\signal{x}\mapsto \signal{y}\in\prox_{\Omega_{\signal{w}}} (\signal{x})$
of $\prox_{\Omega_{\signal{w}}}$ 
will be referred to as ROWL shrinkage.

Note that the set
$\{ (\signal{x} - \signal{w})_+,
(\signal{x} -\signal{w}_{\downarrow})_+ \}
\subset \real^2$ 
in \eqref{eq:prox_rowl} is discrete.
This is similar to the case of 
$\prox_{\norm{\cdot}_0}$ in \eqref{eq:prox_g}.
Thus, resembling the relation between $\prox_{g_\tau}$ and
$\prox_{\widetilde{g}_\tau}$ corresponding to (i) and (ii) of
Fig.~\ref{fig:shrinkage}(a), respectively,
we replace the discrete set to its closed convex hull\footnote{
For a set $\{\signal{a},\signal{b}\}\subset\real^2$, its closed convex
hull is given by $\{\alpha\signal{a} + (1-\alpha)\signal{b}\mid
\alpha\in [0,1] \}$.
}
$\overline{{\rm conv}}
\{ (\signal{x} - \signal{w})_+,
(\signal{x} -\signal{w}_{\downarrow})_+ \}$.
This replacement yields the set-valued operator
$\mathsf{R}: \real^2\rightarrow 2^{\real^2}:
\signal{x} \mapsto
\{{\rm sgn}(\signal{x}) \odot \signal{y} \mid
\signal{y}\in \mathsf{R}(\abs{\signal{x}}) \}$,
where for $\signal{x}\in\real_+^2$
\begin{equation}
\hspace*{-1em} \mathsf{R} (\signal{x}) \!=\!
\begin{cases}
\{ (\signal{x} -
 \signal{w} )_+\}, & 
\hspace*{-.5em}\mbox{if } \!x_1 \!>\! x_2,  \\
\overline{{\rm conv}}\{ (\signal{x} - \signal{w} )_+ ,(\signal{x} - \signal{w}_{\downarrow} )_+ \}, & 
\hspace*{-.5em}\mbox{if } \!x_1 \!=\! x_2,  \\
\{ (\signal{x} - \signal{w}_{\downarrow} )_+\}, & 
\hspace*{-.5em} \mbox{if } \!x_1 \!<\! x_2.  \\
\end{cases}
\label{eq:prox_rowl2}
\end{equation}

As expected, it can be shown that
$\mathsf{R} = \prox_{\widetilde{\Omega}_{\signal{w}}}$ for 
the 1-weakly convex function\footnote{One may add $\signal{w}^\top
\signal{w}/4$ to $\widetilde{\Omega}_\signal{w}$ to make the minimum value
be zero.}
(see Figs.~\ref{fig:rowl_cont} and \ref{fig:rowl_cont_w10})
\begin{align*}
\widetilde{\Omega}_{\signal{w}}&(\signal{x}):=
(\Omega_{\signal{w}}+ (1/2) \norm{\cdot}^2)^{**}(\signal{x}) - (1/2) \norm{\signal{x}}^2
\\
=&
\begin{cases}
\! \signal{w}^\top \! \abs{\signal{x}}_{\downarrow}, & 
\!\!\mbox{if }\!
 \abs{\signal{x}}_{\downarrow} \in \mathcal{K}_1,\\
%
\!\signal{w}^\top \! \abs{\signal{x}}_{\downarrow} 
-\frac{1}{2}
(\abs{\signal{x}}_{\downarrow} \!+\! \signal{w})^\top 
   \signal{B} 
(\abs{\signal{x}}_{\downarrow} \!+\! \signal{w}), 
&
\!\! \mbox{if }\!
 \abs{\signal{x}}_{\downarrow}\in \mathcal{K}_2,\\
w_1\signal{1}^\top \abs{\signal{x}} + \frac{1}{2}\abs{\signal{x}}^\top
 \signal{C}\abs{\signal{x}},
&
\!\! \mbox{if }\!
 \abs{\signal{x}}_{\downarrow}\in \mathcal{K}_3,
\end{cases}
\end{align*}
where
$\signal{B}:= \dfrac{1}{2}\left[
\!
\begin{tabular}{rr}
\! $1$\!\! & $-1$\\
\!$-1$\!\! & $1$
\end{tabular}
\!
\right]
=
\signal{V}
\left[\!
\begin{tabular}{rr}
 $1$\!\! & $0$ \\
$0$\!\! & $0$
\end{tabular}
\!
\right]
\signal{V}^\top
$
with
$\signal{V}:=\dfrac{1}{\sqrt{2}}
\left[
\!
\begin{tabular}{rr}
 $1$ & $1$ \\
$-1$ & $1$
\end{tabular}
\!
\right]
$,
$\signal{C}:= \left[
\!
\begin{tabular}{rr}
\! $0$\!\! & $1$\\
\!$1$\!\! & $0$
\end{tabular}
\!
\right]
=
\signal{V}^\top
\left[\!
\begin{tabular}{rr}
 $1$\!\! & $0$ \\
$0$\!\! & $-1$
\end{tabular}
\!
\right]
\signal{V}
$,
$\mathcal{K}_1:=\{\signal{x}\in\real_{+\downarrow}^2\mid x_1 \geq w_2 -
w_1 + x_2\}$,
$\mathcal{K}_2:= \real_{+\downarrow}^2 \setminus \mathcal{K}_1
= \{\signal{x}\in\real_{+\downarrow}^2\mid 
w_2-w_1 - x_2\leq  x_1 < w_2 - w_1 +x_2\}$,
and
$\mathcal{K}_3:= \real_{+\downarrow}^2 \setminus \mathcal{K}_1
= \{\signal{x}\in\real_+^2\mid  x_1 < w_2-w_1 - x_2\}$.

 \begin{figure}[t!]
 \centering
\begin{tabular}{cc}
\subfigure[eROWL]{
\hspace*{-1em}\includegraphics[height=3.2cm]{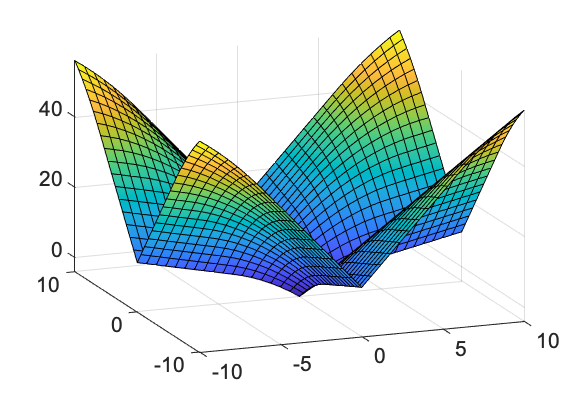}
}
 &
\subfigure[eROWL]{
\hspace*{-1em}\includegraphics[height=3cm]{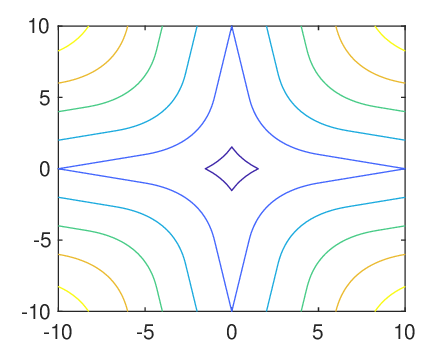}
}
 \\
\subfigure[ROWL]{
\hspace*{-1em}\includegraphics[height=3.2cm]{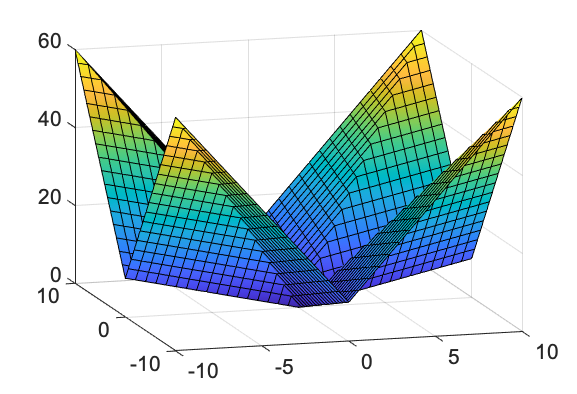}
}
 &
\subfigure[ROWL]{
\hspace*{-1em}\includegraphics[height=3cm]{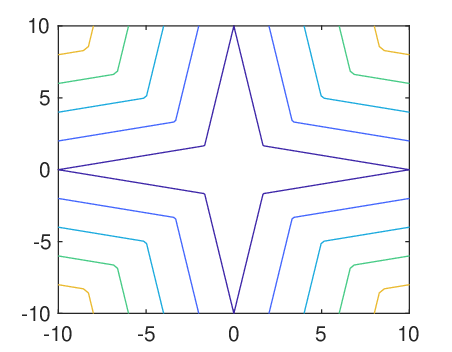}
}
\end{tabular}

  \caption{Surface and contours of eROWL and ROWL
  for the weight vector $\signal{w}=[1,5]^{\top}$.
}
 \label{fig:rowl_cont}
 \end{figure}

We now derive the extended ROWL (eROWL) shrinkage operator by
resembling the relation between
$\prox_{\widetilde{\norm{\cdot}}_0}$ 
and
$\sprox_{\widetilde{\norm{\cdot}}_0/(\delta+1)}$
corresponding to (ii) and (vi) of
Fig.~\ref{fig:shrinkage}, respectively.
The eROWL shrinkage operator for the relaxation parameter $\delta\in\real_{++}$
 is defined by
\begin{equation}
 R_{\delta} := \sprox_{\widetilde{\Omega}_{\signal{w}}/(\delta+1)},
\end{equation}
which is $(1+1/\delta)$-Lipschitz continuous
(see Theorem \ref{theorem:prox_conversion}).
It can readily be verified that
$\signal{p}:= R_{\delta} (\signal{x}) = \left[\mathsf{R}^{-1} + \delta \id \right]^{-1} ((\delta +1)\signal{x})
\Leftrightarrow
(\delta +1)\signal{x} \in
\mathsf{R}^{-1}(\signal{p}) + \delta \signal{p}$,
which gives the geometric interpretation shown in
Fig.~\ref{fig:operator_T}.
Using this, we can verify that
$R_{\delta} (\signal{x}) = 
{\rm sgn}(\signal{x}) \odot R_{\delta}(\abs{\signal{x}})$,
where, for $\signal{x}\in\real_+^2$, 
\begin{align}
&\hspace*{0em} R_{\delta}(\signal{x}) =
\left\{
\begin{array}{l}
\big(\signal{x} - \frac{1}{\delta+1} \signal{w} \big)_+,
 ~\mbox{if }\! x_1 \geq x_2, \signal{x}\not\in
\mathcal{C}_1\cup \mathcal{C}_2
\\			
\big(\signal{x} - 
\frac{1}{\delta+1}
\signal{w}_{\downarrow}
\big)_+,
 ~\mbox{if } 
x_1 < x_2, \signal{x}\not\in
\mathcal{C}_1\cup \mathcal{C}_2,
\\
\!\!\left[
\begin{array}{c}
m-w_1\\
0
\end{array}
\right]
+
\dfrac{(\delta+1)x_2 - m}{\delta}
\left[
\begin{array}{c}
\!-1 \!\\
1
\end{array}
\right]\!,
~\mbox{if } \signal{x}\in \mathcal{C}_1,
\\
\signal{x} - 
\frac{1}{\delta+1}\signal{w}_{\alpha(\signal{x})},
~\mbox{if } \signal{x}\in \mathcal{C}_2.
\end{array}
\right.
\label{eq:eROWL}
\end{align}
Here,
$m:=[(\delta+1)(x_1+x_2) + \delta w_1]/(\delta+2)\in\real_{+}$,
$\signal{w}_{\alpha(\signal{x})}:=
\alpha(\signal{x}) \signal{w} +
(1-\alpha(\signal{x})) \signal{w}_{\downarrow}\in\real_{+}^2$
for
$\alpha(\signal{x}) := 1/2 + (\delta+1)(x_1-x_2)/(2\delta(w_2-w_1))\in
(0,1)$,
\begin{equation}
 \mathcal{C}_1:= 
\mathcal{H}_{\mbox{{\tiny $\backslash$}}}^-  \cap 
\mathcal{H}_{\mbox{{\tiny $\slash$}} 1}^+ \cap 
\mathcal{H}_{\mbox{{\tiny $\slash$}} 2}^+
\end{equation}
 is a triangle given by the intersection of 
three halfspaces
\begin{align*}
\mathcal{H}_{\mbox{{\tiny $\backslash$}}}^- :=&~\{\signal{x}\in\real^2\mid x_1\!+\!x_2 \leq
(w_1\!+\!w_2)/(\delta \!+\! 1) \!+\! w_2 \!-\! w_1 \},\\
\mathcal{H}_{\mbox{{\tiny $\slash$}} 1}^+ :=&~
\{\signal{x}\in\real^2\mid -x_1 + (\delta +1)x_2 > \delta w_1/(\delta+1)\},\\
\mathcal{H}_{\mbox{{\tiny $\slash$}} 2}^+ :=&~
\{\signal{x}\in\real^2\mid  (\delta +1)x_1 -x_2 > \delta
w_1/(\delta+1)\},
\end{align*}
and
\begin{equation}
 \mathcal{C}_2:= \mathcal{H}_{\mbox{{\tiny $\backslash$}}}^+ \cap \mathcal{S}
\end{equation}
is an unbounded set given by the intersection of
the hyperslab and the halfspace
\begin{align}
\hspace*{-.5em}\mathcal{S} :=&~ \{\signal{x}\in\real^2\mid \abs{x_1-x_2} <
 \delta(w_2-w_1)/(\delta+1) \}
\label{eq:setS}\\
\hspace*{-.4em} \mathcal{H}_{\mbox{{\tiny $\backslash$}}}^+ \!:=&~\{\signal{x}\in\real^2\mid x_1 \!+\! x_2 >
(w_1 \!+\! w_2)/(\delta+1) + w_2-w_1 \} \nonumber\\
\hspace*{-.5em}=&~ \real^2\setminus
\mathcal{H}_{\mbox{{\tiny $\backslash$}}}^-.
\label{eq:setHbackslash+}
\end{align}

 \begin{figure}[t!]
 \centering
\begin{tabular}{cc}
\subfigure[eROWL]{
\hspace*{-1em}\includegraphics[height=3.2cm]{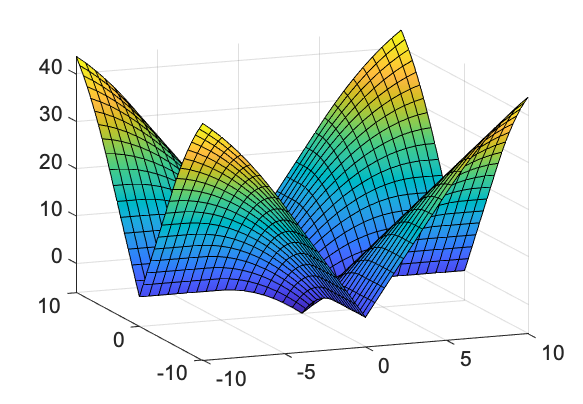}
}
 &
\subfigure[eROWL]{
\hspace*{-1em}\includegraphics[height=3cm]{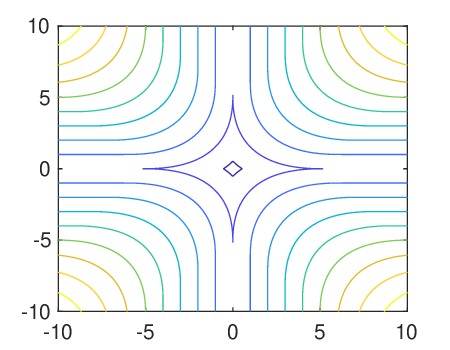}
}
 \\
\subfigure[ROWL]{
\hspace*{-1em}\includegraphics[height=3.2cm]{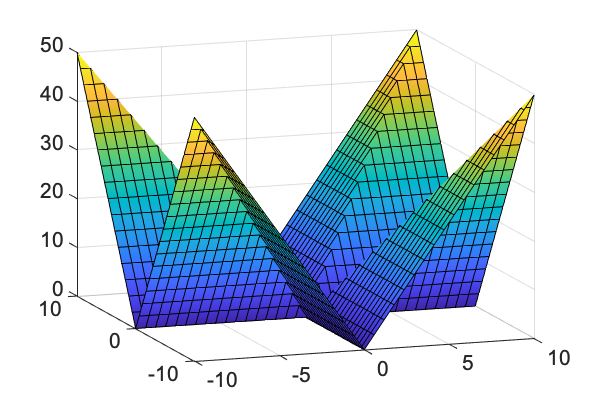}
}
 &
\subfigure[ROWL]{
\hspace*{-1em}\includegraphics[height=3cm]{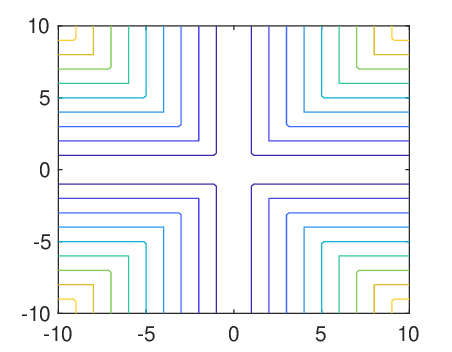}
}
\end{tabular}

  \caption{Surface and contours of eROWL and ROWL
  for the weight vector $\signal{w}=[0,5]^{\top}$.
}
 \label{fig:rowl_cont_w10}
 \end{figure}

For every $\signal{x}\in\real^2$,
it holds that $\lim_{\delta\downarrow 0} R_{\delta}(\signal{x})\in
\mathsf{R}(\signal{x})$,
where 
$\lim_{\delta\downarrow 0} R_{\delta}(\signal{x}) = \signal{x} -
\signal{w}_{1/2}$
over $\{\signal{x}\in\real_+^2\mid x_1=x_2 \}$.
%
An arbitrary selection $U:\real^2\rightarrow \real^2$ of 
the set-valued operator
$\mathsf{R}(= \prox_{\widetilde{\Omega}_{\signal{w}}})$
jointly satisfies (i) range $U\neq \real^2$ and
(ii) $\widetilde{\Omega}_{\signal{w}} +
(1/2)\norm{\cdot}^2\in\Gamma_0(\real^2)$.
This gives a counterexample where $2) \not\Rightarrow 1)$ in Proposition 
\ref{proposition:surjectivity_and_continuityweakconvexity}.
(The same applies to a selection of $\prox_{\widetilde{\norm{\cdot}_0}}$.)
On the other hand, it holds that range $R_{\delta}=\real^2$,
as consistent with Proposition
\ref{proposition:surjectivity_and_continuityweakconvexity}.
The operator $R_{\delta}$ is a MoL-Grad denoiser;
this is a direct consequence of Theorem \ref{theorem:prox_conversion} and 
Fact \ref{fact:weaklyconvex_necsuffcondition}.
\begin{corollary}
\label{corollary:erowl_mol}
For every $\delta\in\real_{++}$,
$R_{\delta} = \sprox_{\widetilde{\Omega}_{\signal{w}}/(\delta+1)}$
 can be expressed as the 
($1+1/\delta$)-Lipschitz continuous gradient
of a differentiable convex function.
\end{corollary}

 \begin{figure}[t!]
   \psfrag{M}[Bc][Bc][1]{$\mathcal{M}:={\rm span}\{[1,1]^\top\}$}
   \psfrag{Tm}[Bc][Bc][1]{$\mathsf{R}(\signal{m})$}
   \psfrag{m}[Bl][Bl][1]{$\signal{m}=\mathsf{R}^{-1}(\signal{p})$}
   \psfrag{=}[Bl][Bl][1]{$= \mathsf{R}^{-1}(\hat{\signal{p}})$}
   \psfrag{x}[Br][Br][1]{$\signal{x}$}
   \psfrag{x1}[Bl][Bl][1]{$x_1$}
   \psfrag{x2}[Br][Br][1]{$x_2$}
   \psfrag{p}[Bl][Bl][1]{\hspace*{-4em}$\signal{p} = R_{\delta}(\signal{x})$}
   \psfrag{xt}[Br][Br][1]{$\hat{\signal{x}}$}
   \psfrag{pt}[Br][Br][1]{$\hat{\signal{p}}$}
   \psfrag{pt}[Br][Br][1]{$\hat{\signal{p}} = R_{\delta}(\hat{\signal{x}})$}
   \psfrag{distrho}[Bl][Bl][.8]{$\dfrac{(w_1+w_2)\delta}{\sqrt{2}(\delta+1)}$}
   \psfrag{dist1}[Bl][Bl][.8]{$\dfrac{w_1+w_2}{\sqrt{2}(\delta+1)}$}
   \psfrag{dist2rho}[Br][Br][.8]{$\dfrac{(w_2- w_1)\delta}{\sqrt{2}(\delta+1)}$}
   \psfrag{dist21}[Br][Br][.8]{$\dfrac{w_2 - w_1}{\sqrt{2}(\delta+1)}$}
   \psfrag{w2}[Br][Br][1]{$w_2-w_1$}
   \psfrag{w2delta}[Br][Br][1]{$\frac{(w_2-w_1)\delta}{\delta+1}$}

 \centering
\hspace*{.4cm}\includegraphics[height=5cm]{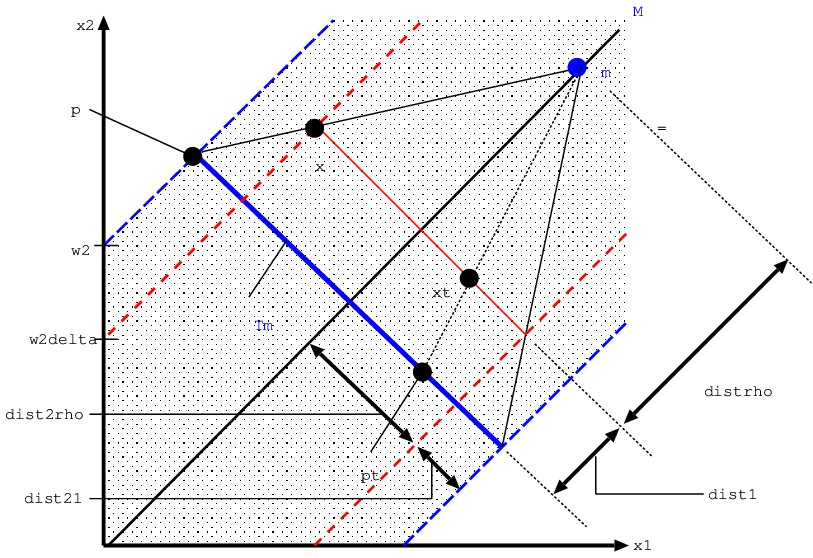}

  \caption{A geometric interpretation of conversion from $\mathsf{R}$ to
  $R_{\delta}$.}
\vspace*{-.5em}
 \label{fig:operator_T}
 \end{figure}


Figure \ref{fig:rowlshrinkage} illustrates 
the ROWL and eROWL shrinkage operators, 
where Figs.~\ref{fig:rowlshrinkage}(a) and \ref{fig:rowlshrinkage}(d)
depict ROWL shrinkage and eROWL shrinkage as a vector field, respectively,
while
Figs.~\ref{fig:rowlshrinkage}(b), \ref{fig:rowlshrinkage}(c),
\ref{fig:rowlshrinkage}(e), and \ref{fig:rowlshrinkage}(f)
depict (a part of) each graph of those operators.
More specifically, 
Figs.~\ref{fig:rowlshrinkage}(b) and \ref{fig:rowlshrinkage}(e) illustrate
the graphs of $R_{\rm ROWL}$ and $R_{\delta}$, respectively, 
on the two dimensional plane, where
the solid lines depict the images of line segments (dot-dashed line).
On the other hand, 
Figs.~\ref{fig:rowlshrinkage}(c) and \ref{fig:rowlshrinkage}(f)
depict the graphs of the first components of 
$R_{\rm ROWL}$ and $R_{\delta}$, respectively.
In Figs.~\ref{fig:rowlshrinkage}(a), \ref{fig:rowlshrinkage}(b),
\ref{fig:rowlshrinkage}(d), and \ref{fig:rowlshrinkage}(e),
the dotted line on the diagonal is the set of $\signal{x}$'s with
$x_1=x_2$.

As seen from 
Figs.~\ref{fig:rowlshrinkage}(a) -- \ref{fig:rowlshrinkage}(c), 
ROWL shrinkage shows discontinuity on the diagonal.
In Fig.~\ref{fig:rowlshrinkage}(a), there exist only two types of flow
basically (excluding the neighborhood of the axes),
and all flows are identical in each side of the diagonal.
To see the behaviour around the diagonal, let us turn attention to
Fig.~\ref{fig:rowlshrinkage}(b), where the graphs 
${\rm gra}R_{\rm ROWL}:=
\{(\signal{x}, R_{\rm ROWL}(\signal{x})) \mid
\signal{x}\in\real^2\} \subset \real^2\times \real^2$
are depicted on $\real^2$.
Here, $R_{\rm ROWL}(\signal{x})$'s corresponding to 
those $\signal{x}$'s on dot-dashed line are plotted by solid line;
circle on the solid line corresponds to 
each square on the dot-dashed line.
It can be seen that $R_{\rm ROWL}(\signal{x})$ jumps from a blue point
on the line to (the vicinity of) a red open circle on the doted line
when $\signal{x}$ crosses the diagonal, and thus 
each dot-dashed line is mapped to 
two separate segments (depicted by the solid lines with circles on them).
This illustrates the discontinuity of ROWL shrinkage,
which typically causes difficulty in analyzing
convergence when the operator is used in the splitting algorithms.
The discontinuity of ROWL shrinkage can also be seen
from Fig.~\ref{fig:rowlshrinkage}(c).

In stark contrast, eROWL shrinkage is continuous,
as illustrated in Figs.~\ref{fig:rowlshrinkage}(d) -- \ref{fig:rowlshrinkage}(f).
In Figs.~\ref{fig:rowlshrinkage}(d), 
the flow in the neighbor of the diagonal
(specifically, in the region between two doted lines parallel to the
diagonal) changes continuously as $\signal{x}$ approaches the diagonal.
As $\signal{x}$ becomes closer to the diagonal, 
the flow becomes more parallel to the diagonal line
so that it is exactly parallel on the diagonal.
The continuous change is described  better in
Figs.~\ref{fig:rowlshrinkage}(e) and \ref{fig:rowlshrinkage}(f);
in Fig.~\ref{fig:rowlshrinkage}(e),
 each dot-dashed line is mapped to a continuous polygonal line
(a piecewise-linear curve).
The continuity is the remarkable difference (from ROWL) of great importance.


 \begin{figure*}
   \psfrag{When x crosses the diagonal...}[Bc][Bc][.8]{\hspace*{5em}When $\signal{x}$ crosses the diagonal...}
   \psfrag{Rw(x) jumps discontinuously.}[Bc][Bc][.8]{\hspace*{-2em}$R_{\rm ROWL}(\signal{x})$ jumps discontinuously.}
   \psfrag{Rdelta(x) changes continuously.}[Bc][Bc][.8]{\hspace*{-3em}$R_{\delta}(\signal{x})$ changes
  continuously.}
   \psfrag{x1}[Bc][Bc][.8]{$x_1$}
   \psfrag{x2}[Bc][Bc][.8]{$x_2$}
 \centering

\hspace*{-1em}
\begin{tabular}{ccc}
\subfigure[ROWL: vector field]{
\includegraphics[height=4cm]{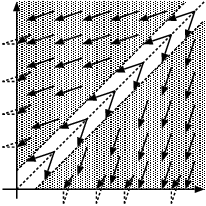}
}
 &\hspace*{1em}
\subfigure[ROWL: a part of the graph]{
\includegraphics[height=4cm]{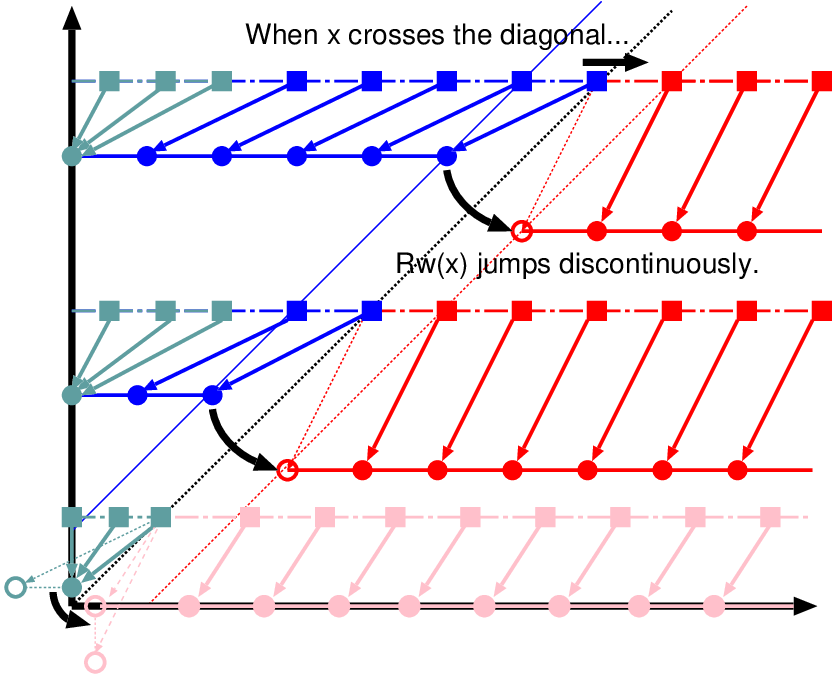}
}

 &
\subfigure[ROWL: graph of $\signal{x}\mapsto (R_{\rm ROWL}(\signal{x}))_1$]{
\includegraphics[height=4cm]{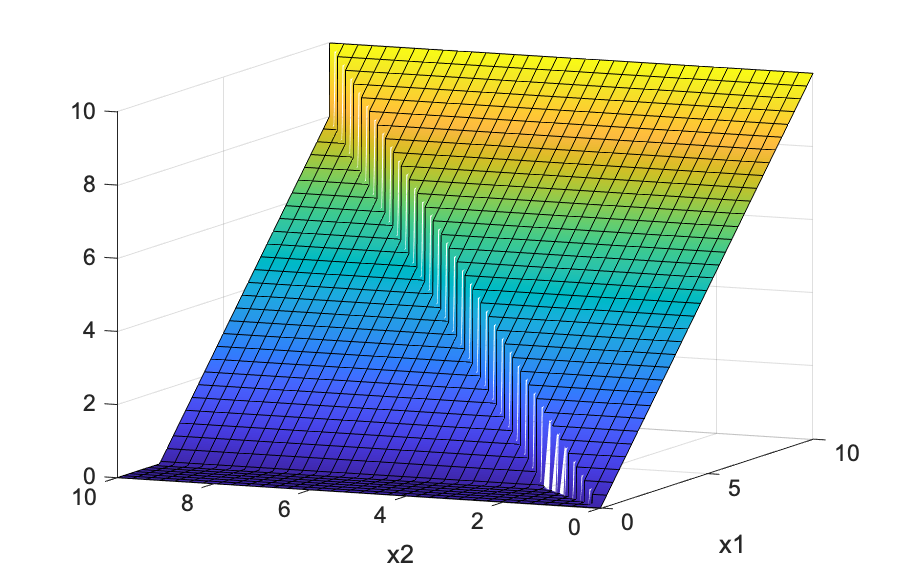}
}

\end{tabular}

\begin{tabular}{ccc}
\subfigure[eROWL vector field]{
\includegraphics[height=4cm]{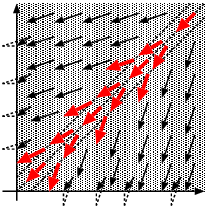}
}
 &
\subfigure[eROWL: a part of the graph]{
\includegraphics[height=4cm]{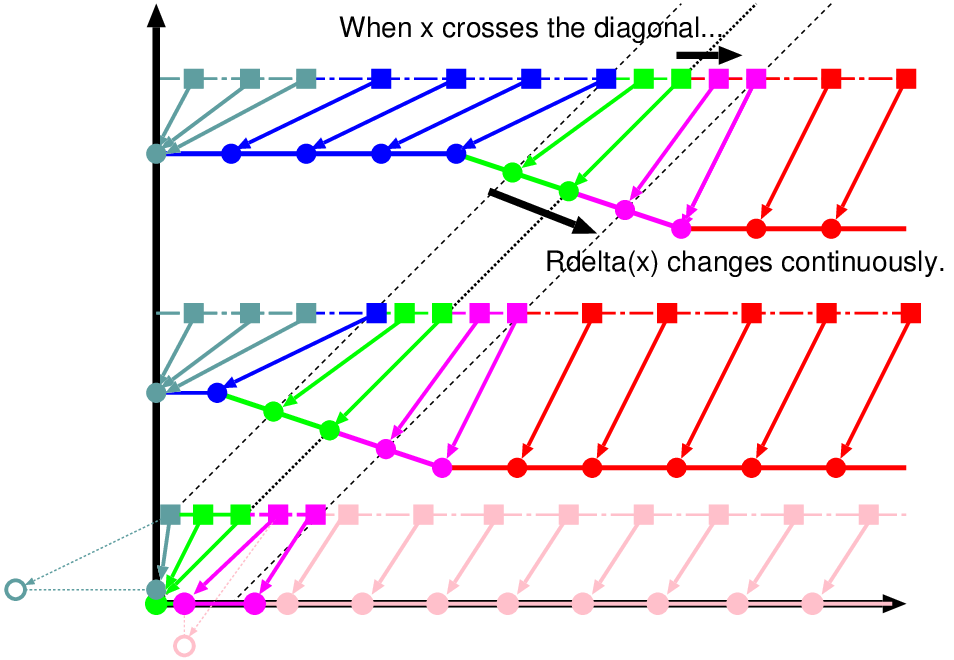}
}

 &
\subfigure[eROWL: graph of $\signal{x}\mapsto (R_{\delta}(\signal{x}))_1$]{
\includegraphics[height=4cm]{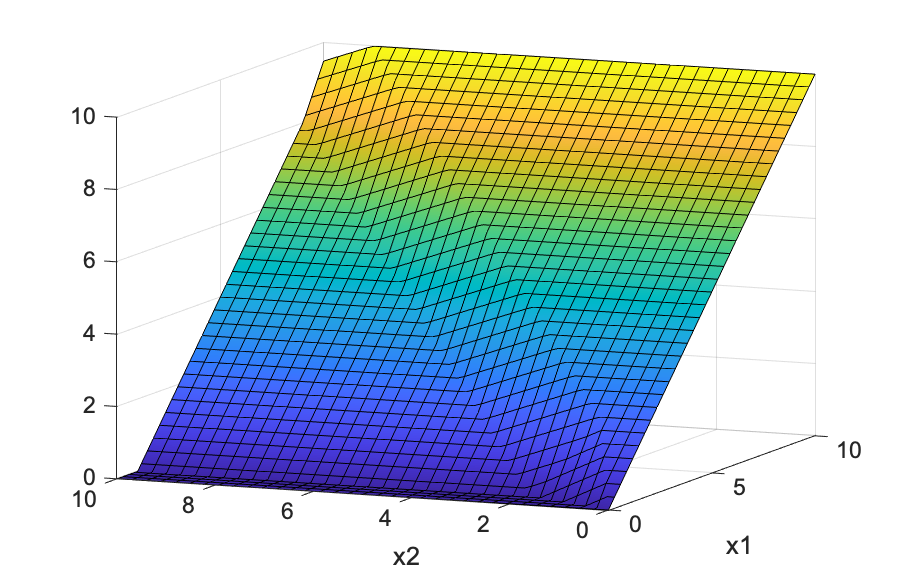}
}

\end{tabular}

\caption{Illustrations of ROWL and eROWL shrinkage operators as vector
 field and as graph. 
In (a) and (d),
the shaded region indicates the range of operator.
In (b) and (e), square ($\signal{x}$) on the dot-dashed lines
is mapped to circle
($R_{\rm ROWL}(\signal{x})$ or $R_{\delta}(\signal{x})$) on
the solid lines which therefore represent (a part of) the graphs of the operators.
In (c) and (f), the graphs of the functions
$\signal{x}\mapsto (R_{\rm ROWL}(\signal{x}))_1$ and 
$\signal{x}\mapsto (R_{\delta}(\signal{x}))_1$ 
returning the first component
are depicted, respectively, for $\signal{w}:=[0,2]^{\top}$ and $\delta:=1$.
}
 \label{fig:rowlshrinkage}
 \end{figure*}


\begin{remark}
\label{remark:weights}
The eROWL shrinkage operator $R_{\delta}$
given in \eqref{eq:eROWL}
can be represented with another set of parameters
$\widetilde{\signal{w}}_{\rm eROWL}:=\signal{w}/(\delta+1)$ and
$\varpi:=(w_2-w_1)\delta/(\delta+1)$,
in place of $\signal{w}$ and $\delta$.
Comparing \eqref{eq:prox_rowl2} and \eqref{eq:eROWL}
under this representation, one can see that
$\widetilde{\signal{w}}_{\rm eROWL}$ of
$R_{\delta}$ (eROWL) corresponds to 
$\signal{w}_{\rm ROWL}$ (the weight vector $\signal{w}$ of
$\mathsf{R}$), and 
it will therefore be fair to set the weight vectors
in a way that $\widetilde{\signal{w}}_{\rm eROWL}=\signal{w}_{\rm ROWL}$.
The parameter $\varpi$ gives the bound of $\abs{x_1-x_2}$ in the
definition of 
$\mathcal{S}$ given in \eqref{eq:setS}. 
\end{remark}

\begin{remark}[Relation to WL1L0 regularization]
\label{remark:prior_work}
Recently, the joint weighted $\ell_1$ and $\ell_0$-norm (WL1L0) regularization
{\rm \cite{berkessa24}} has been proposed:
\begin{equation}
 \min_{\signal{c},\signal{d}\in\real^n} \norm{\signal{y}- \signal{A}
  (\signal{c}+ \signal{d})}_2^2 + 
\lambda(
\alpha \norm{\signal{c}}_1 + (1-\alpha)\norm{\signal{d}}_0
),
\label{eq:wl1l0}
\end{equation}
where $\lambda\in\real_{++}$ and $\alpha\in(0,1)$.
Here, $\signal{y} = \signal{A}\signal{x}_{\diamond} +\signal{\epsilon}$
is the output vector for a given matrix $\signal{A}$,
the estimand $\signal{x}_{\diamond}$, and the noise vector $\signal{\epsilon}$.
As the WL1L0 regularization is seemingly related
to ROWL shrinkage, let us inspect \eqref{eq:wl1l0} by slightly
 reformulating it as follows:
\begin{align}
& \min_{\signal{x},\signal{c},\signal{d}\in\real^n} \norm{\signal{y}- \signal{A}
  \signal{x}}_2^2 + 
\lambda(
\alpha \norm{\signal{c}}_1 + (1-\alpha)\norm{\signal{d}}_0
)\nonumber\\
& \mbox{s.t.}~\signal{x} = \signal{c}+\signal{d}.
\label{eq:wl1l0_reformulation}
\end{align}
Assume that 
\eqref{eq:wl1l0_reformulation} has a solution
$(\signal{x}^*,\signal{c}^*,\signal{d}^*)$.
Then, the $i$th components,
$i\in\{1,2,\cdots,n\}$,  of $\signal{c}^*$ and $\signal{d}^*$
are given by
$c_i^*=x_i^* 1(\abs{x_i^*} <(1-\alpha)/\alpha)$ and
$d_i^*=x_i^* 1(\abs{x_i^*} \geq (1-\alpha)/\alpha)$, 
respectively.
Hence, $\signal{x}^*$ can be characterized as a solution of the
following problem:
\begin{align}
& \min_{\signal{x}\in\real^n} \norm{\signal{y}- \signal{A}
  \signal{x}}_2^2 + 
\lambda \Omega_{{\rm Capped\mbox{-}}\ell_1}(\signal{x}),
\label{eq:cappedl1}
\end{align}
where $\Omega_{{\rm Capped\mbox{-}}\ell_1}(\signal{x}):=
\sum_{i=1}^{n}\phi_{{\rm Capped\mbox{-}}\ell_1}(x_i)$
with
$\phi_{{\rm Capped\mbox{-}}\ell_1}:
\real\rightarrow\real:x\mapsto \min\{\alpha\abs{x},1-\alpha\}$
which is the capped-$\ell_1$ penalty {\rm \cite{tzhang08}}
essentially.

The WL1L0 (capped-$\ell_1$) regularizer penalizes inactive components
by the $\ell_1$ norm, while it gives no penalty to active components
in the sense that an increase of those active components preserves 
the penalty.
This applies to the case of ROWL when the zero weights are assigned
to the top component(s) of $\abs{\signal{x}}_{\downarrow}$ and
some strictly-positive uniform weights are assigned to the others.
Note, however, that there is a significant difference
with respect to the way of identifying the active components.
Indeed, WL1L0 specifies the threshold $(1-\alpha)/\alpha$, and 
all such components having magnitudes above the threshold are identified
to be active.
In contrast, ROWL can specify the number of active components directly
via predetermining the number of zero weights.
This is advantageous in such scenarios when 
the number of active components is {\it a priori} known
(but their magnitudes are not) and/or when 
one wishes to extract a specific number of feature vectors.

In addition to the notable difference discussed in the previous
paragraph, the proposed eROWL shrinkage is a MoL-Grad denoiser 
(see Corollary \ref{corollary:erowl_mol}), which means that
plugging it into the splitting methods gives a promising algorithm
in the sense of generating a sequence converging to a solution of 
the  associated optimization problem.
See {\rm \cite{yukawa_molgrad24}} for more details.
This is a remarkable advantage because the algorithms for 
WL1L0 or capped-$\ell_1$ only have a guarantee of convergence to 
a stationary point which is not necessarily a global minimizer.
\end{remark}

\section{Numerical Examples: Sparse Signal Recovery}
\label{sec:simulation}

We consider the simple linear model
$\signal{y} := \signal{A} \signal{x}_{\diamond} + \signal{\epsilon}$,
where $\signal{x}_{\diamond} \in\real^N$ is the sparse (or weakly sparse) signal,
$\signal{A} \in\real^{M\times N}$ is the measurement matrix, and
$\signal{\epsilon}\in\real^M$ is the i.i.d.~Gaussian noise vector.
To recover the signal $\signal{x}_{\diamond}$,
we consider the iterative shrinkage algorithm in the following form:
\begin{equation}
 \signal{x}_{k+1}:= T\left(
 \signal{x}_k - \mu\nabla f( \signal{x}_k)
\right),~k\in\Natural,
\label{eq:pfbs}
\end{equation}
where 
$T$ is the shrinkage operator (hard, firm, soft, ROWL, or eROWL), and
$f:\real^N\rightarrow\real_+:\signal{x}\mapsto
\frac{1}{2}\norm{\signal{A} \signal{x} - \signal{y}}_2^2$
is the squared-error function with the step size parameter
$\mu\in\real_{++}$.
Clearly, the function $f$ is $\rho$-strongly convex with
$\kappa$-Lipschitz continuous gradient $\nabla f: \signal{x}\mapsto 
\signal{A}^\top (\signal{A}\signal{x}-\signal{y})$
for
$\rho := \lambda_{\min}(\signal{A}^\top \signal{A})$ and 
$\kappa := \lambda_{\max}(\signal{A}^\top \signal{A})$.
As an evaluation metric, we adopt the system mismatch
$\norm{\signal{x}_{\diamond} -\signal{x}_k}_2^2/
\norm{\signal{x}_{\diamond}}_2^2$.

\begin{remark}[Prior works of nonconvex penalties]
 There is a vast amount of literature on nonconvex penalties
for sparse signal recovery (or sparse modeling);
the examples include
the $\ell_p$ quasi-norm
{\rm \cite{chartrand07,marjanovic12,jeong14,yukawa16}},
log-sum function {\rm \cite{candes08}}, 
capped $\ell_1$ {\rm \cite{zhang09}},
MC {\rm \cite{zhang}},
smoothly clipped absolute deviation (SCAD) {\rm \cite{fan01}}, 
continuous exact $\ell_0$ (CEL0) {\rm \cite{soubies15}}, 
sorted concave penalty \cite{long19},
to name a few.
For most of those nonconvex penalties, 
the (possibly set-valued) proximity operators are listed 
in the prox repository\footnote{http://proximity-operator.net/nonconvexfunctions.html}.

Of closer relation to the present study is those based on the notion of
``convexity-preservation'' {\rm \cite{blake87,nikolova99}}, which means that
the penalty is nonconvex (precisely, weakly convex) but
the whole cost function remains convex owing to the strong
 convexity of the loss/fidelity function or other terms.
It is related to the classical notion of difference of convex (DC)
 programming {\rm \cite{dinh86,parekh16,lanza19}}.
In the recent works, this line of research is referred to as
convex-nonconvex strategy or the generalized Moreau enhanced model
{\rm \cite{selesnick,lanza21,abe_ip20}}.
Further developments/extension can be found in
{\rm \cite{alshabili21,yata22,yzhang23,yukawa23}}.

This idea has been exploited in robust signal recovery as well
{\rm \cite{suzuki20,suzuki23,yukawa23,tillmann24}}.
A systematic way of enhancing the proximity operator 
by taking an affine combination of two proximity operators with
positive- and negative-valued weights has been proposed
under the name of external division operator {\rm \cite{suzuki_icassp24}}.
In the present study, we focus on the efficacy of 
the continuous relaxation which is the main subject, and 
comparisons to the other nonconvex penalties are left
as future works.

\end{remark}

\subsection{Hard Shrinkage versus Firm Shrinkage}

We compare the performance of the ``discontinuous'' hard shrinkage operator
with its continuous relaxation which is firm shrinkage.
For comparison, we also test soft shrinkage.
The signal $\signal{x}_{\diamond} \in\real^N$ of dimension $N:=50$ is
generated as follows:
the first $s$ components are generated from
the i.i.d.~standard Gaussian distribution $\mathcal{N}(0,1)$, and
the other $N-s$ components are generated from
i.i.d.~$\mathcal{N}(0,1.0\times 10^{-4})$.
The matrix $\signal{A}$ is generated also from
$\mathcal{N}(0,1)$.
We study the impacts of the parameter $\tau_1$ of firm shrinkage and the threshold
$\tau$ of soft/hard shrinkage on the performance.
Although $\mu$ and $\tau_2$ can be tuned within the range given in
\cite[Theorem 2]{yukawa_molgrad24}, those parameters are set systematically to 
$\mu:= (2-\varepsilon)/(\kappa+\rho)$ and 
$\tau_2:= \tau_1/(\mu\rho)$ for $\varepsilon:=1.0\times 10^{-6}$
(see Appendix \ref{appendix:firm_parameter}).
The step size of soft/hard shrinkage is set to $1/\kappa\in(0,2/\kappa)$.
The results are averaged over 20,000 independent trials.

Figure \ref{fig:simulation1} depicts the results
for three different sparsity levels $s:=5$, 10, 20
under $M:=100$, 200 and
the signal-to-noise ratio (SNR) $10,20$ dB.
Table \ref{table:reduction_rate} summarizes the reduction rate (gain) of
firm shrinkage against hard/soft shrinkage.
Thanks to the continuous relaxation,
firm shrinkage gains 19.5--44.9 \% against hard shrinkage.
Compared to soft shrinkage, the gain is 5.3--37.9 \%.
%











 \begin{figure}[t!]
\psfrag{t}[Bc][Bc][.8]{$\tau$}

 \centering

\subfigure[$M=100$, SNR$=10$ dB]{
\includegraphics[height=4cm]{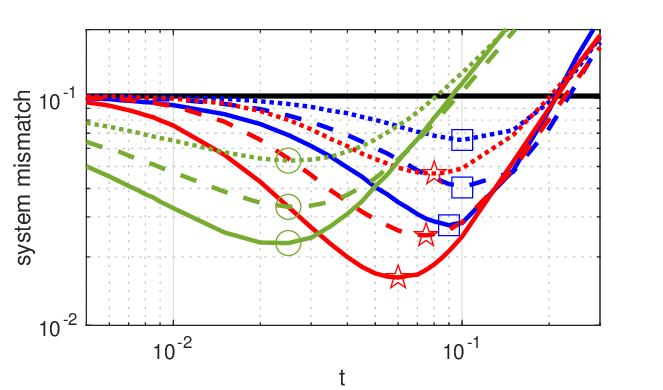}
}\vspace*{-1.1em}

\subfigure[$M=200$, SNR$=10$ dB]{
\includegraphics[height=4cm]{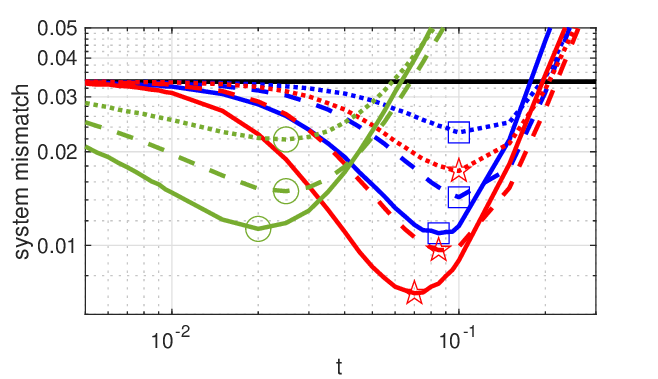}
}\vspace*{-1.1em}

\subfigure[$M=100$, SNR$=20$ dB]{
\includegraphics[height=4cm]{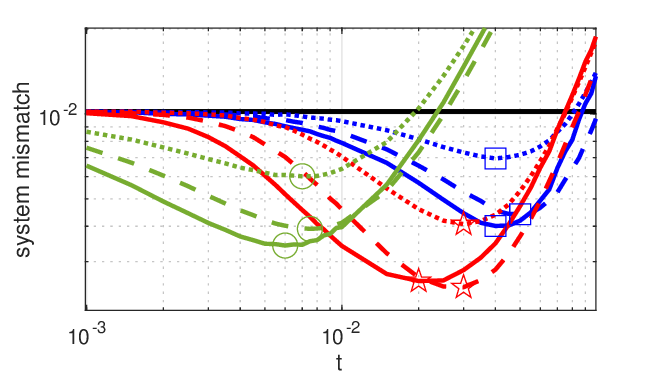}
}\vspace*{-1.1em}

\subfigure[$M=200$, SNR$=20$ dB]{
\includegraphics[width=7cm]{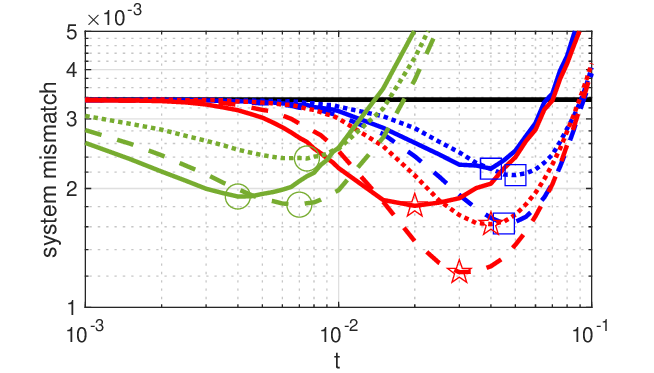}
}

  \caption{Comparisons of the iterative shrinkage algorithms
with hard shrinkage (blue), 
firm shrinkage (red), 
and soft shrinkage (green), where
the horizontal axis for firm shrinkage corresponds to $\tau_1$.
Solid, dashed, and dotted curves correspond to
sparsity level $s:=5$, $s:=10$, and $s:=20$, respectively.
Mark (square, pentagon, and circle) indicates
the best point (the lowest system mismatch) on each curve.
}
\vspace*{-1em}
 \label{fig:simulation1}
 \end{figure}


\subsection{ROWL Shrinkage versus eROWL Shrinkage}
\label{subsec:rowl_erowl_simulation}

The operator $R_{\delta}$ is Lipschitz continuous
with constant $(1-\eta)^{-1} = (1+1/\delta)$
(see Theorem \ref{theorem:prox_conversion}),
where $\eta=1/(1+\delta)$.
To exploit \cite[Theorem 2]{yukawa_molgrad24},
let $\beta :=1-\eta = \delta/(1+\delta)$.
Then, the sequence $(\signal{x}_k)_{k\in\Natural}$ generated by
\eqref{eq:pfbs} converges to a minimizer (if exists) of 
$\mu f + \widetilde{\Omega}_{\signal{w}}/(\rho+1)$,
provided that 
(a) $\delta > (\kappa-\rho)/2\rho$
($\Leftrightarrow \beta  >(\kappa-\rho)/(\kappa+\rho)$)
and
(b) $\mu\in [(1-\beta)\rho,(1+\beta)/\kappa)$.
Thus, unless otherwise stated,
we set 
$\delta := \gamma_{\delta}(\kappa - \rho)/(2\rho)$
and
$\mu := \gamma_{\mu} (1-\beta)/\rho + (1-\gamma_{\mu})(1+\beta)/\kappa$
with the additional parameters
$\gamma_{\delta}> 1$ and $\gamma_{\mu} \in(0,1]$
fixed to $\gamma_{\delta}:=1.01$ and $\gamma_{\mu}:=0.5$.
In the following,
eROWL shrinkage $R_{\delta}$ is compared to 
ROWL shrinkage as well as firm shrinkage,
where $R_{\delta}$ in \eqref{eq:pfbs} is replaced by 
those other shrinkage operators.

\begin{table}
\caption{Reduction rate (\%) of firm shrinkage in system mismatch} 
\label{table:reduction_rate}
\centering
\begin{tabular}{c|l|l}
 & against hard shrinkage & against soft shrinkage\\
 & $s=5$~~$s=10$~~$s=20$ & $s=5$~~$s=10$~~$s=20$ \\ \hline
case (a) &
~41.1\hspace*{1.5em}38.5\hspace*{1.8em}28.7&
  ~29.5\hspace*{1.5em}25.1\hspace*{1.8em}12.1\\
case (b) & 
  ~35.9\hspace*{1.5em}32.4\hspace*{1.8em}24.8&
   ~37.9\hspace*{1.5em}35.3\hspace*{1.8em}20.8\\
case (c) & 
   ~36.2\hspace*{1.5em}44.9\hspace*{1.8em}41.5&
   ~25.1\hspace*{1.5em}37.9\hspace*{1.8em}32.3\\
case (d) & 
~19.5\hspace*{1.5em}24.7\hspace*{1.8em}24.9&
    \hspace*{.8em}5.3\hspace*{1.5em}32.5\hspace*{1.8em}31.8\\
\end{tabular}

\end{table}

\subsubsection{Illustrative Examples}
\label{subsubsec:illustrative}

 \begin{figure}[t!]
\psfrag{x1}[Bc][Bc][.8]{$x_1$}
\psfrag{x2}[Bc][Bc][.8]{$x_2$}

\centering
\subfigure[ROWL]{
\includegraphics[height=5cm]{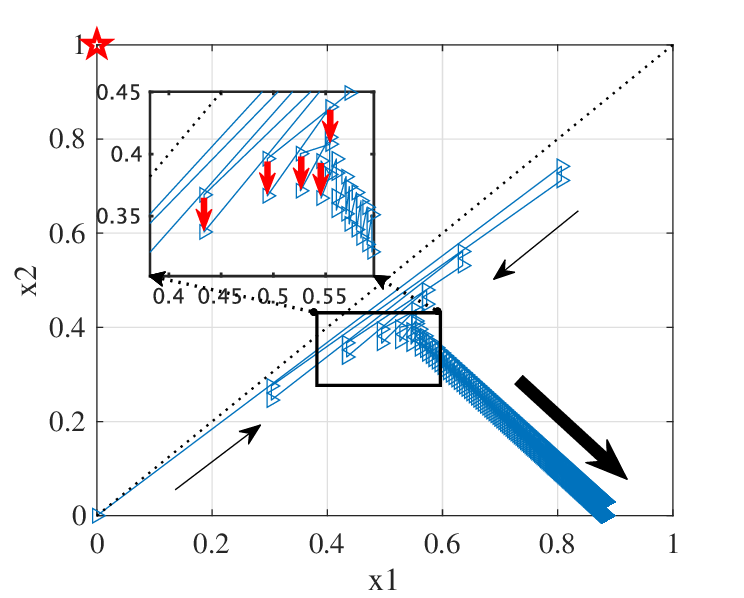}
}\vspace*{-1em}

\subfigure[eROWL]{
\includegraphics[height=5cm]{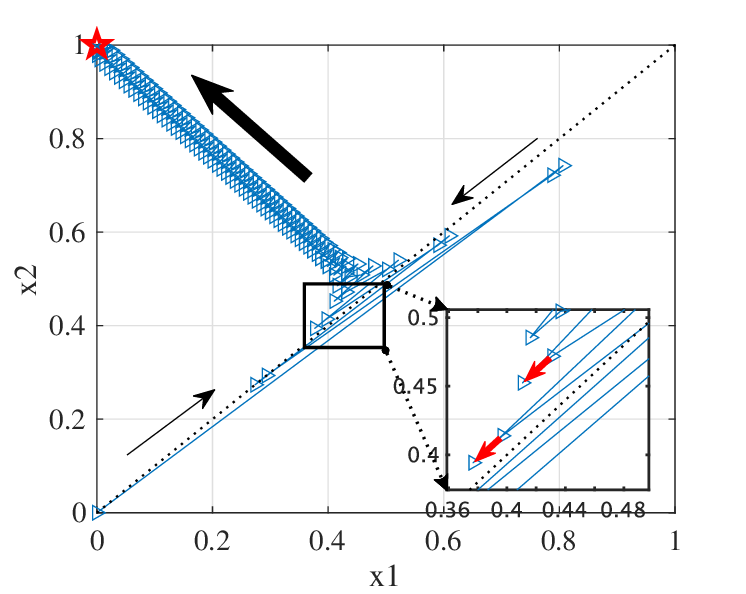}
}

  \caption{Illustrative examples in which ROWL fails but
  eROWL works.
The signal $\signal{x}_{\diamond}$ is depicted by a pentagram.}
\vspace*{-1em}
 \label{fig:simA}
 \end{figure}

Unlike the continuous eROWL shrinkage $R_{\delta}$,
ROWL shrinkage is discontinuous and its corresponding iterate has no
guarantee of convergence.
To illuminate the potential issue,
we consider 
the noiseless situation (i.e., $\signal{\epsilon}:=\signal{0}$)
with
the $2\times 2$ matrix 
$\signal{A}:=
\signal{Q}
\left[
\begin{tabular}{rr}
 $1$\!\! & $0$ \\
$0$\!\! & $0.1$
\end{tabular}
\right]
\signal{Q}^\top
$
(i.e., $M:=2$)
for
$\signal{Q}:= \left[
\begin{tabular}{rr}
 $1$\!\! & $-0.9$ \\
$0.9$\!\! & $1$
\end{tabular}
\right]
$
and the sparse signal
$\signal{x}_{\diamond} = [0,1.0]^\top$.
In this case, $\kappa=0.82$, $\rho=8.2\times 10^{-3}$.
For the sake of illustration,
the step size is set to $\mu:=2.0$ for both ROWL and eROWL,
the weight vector of ROWL is set to
$\signal{w}_{\rm ROWL}:=[0,0.03]^\top$, and
the weight vector $\signal{w}_{\rm eROWL}$ of eROWL is chosen, for
fairness, in such a way that 
$(\widetilde{\signal{w}}_{\rm eROWL}=)\signal{w}_{\rm eROWL}/(\delta+1)
= \signal{w}_{\rm ROWL}$ 
with $\delta=50.0$ (see Remark \ref{remark:weights}).
The convergence for eROWL
is guaranteed for every $\mu\in [1.6\times 10^{-4}, 2.4)$.
The algorithms are initialized to $\signal{x}_0:=\signal{0}$.

Figure \ref{fig:simA} plots the points
$\signal{x}_0,\signal{x}_{0.5},\signal{x}_{1},\signal{x}_{1.5},\signal{x}_{2},\cdots$,
where $\signal{x}_{k+1/2}:= \signal{x}_k - \mu\nabla f( \signal{x}_k)$,
$k\in\Natural$,
is the intermediate vector between $\signal{x}_k$ and $\signal{x}_{k+1}$.
The arrows in red color depict the displacement vectors
$\signal{x}_{k+1} - \signal{x}_{k+1/2}$ visualizing how each shrinkage operator works.
In view of \eqref{eq:prox_rowl},
the displacement vector for ROWL is given by $-\signal{w}$ basically,
guiding the estimate toward a wrong direction.
The ROWL iterate converges numerically to
$\signal{x}_{\infty}^{\rm ROWL}:=[0.88,0]^{\sf T}$, failing to identify
the active component.
In sharp contrast, 
the eROWL iterate converges to
$\signal{x}_{\infty}^{\rm eROWL}:=[0,0.99]^{\sf T}$,
identifying the active component correctly.
This is because 
the vector $\signal{w}_{\alpha}$ governing the displacement vector 
depends on the position of the current estimate.
More precisely, since the estimate at the early phase is located in the
neighborhood of the diagonal (where $x_1=x_2$), we have
$\signal{w}_{\alpha} \approx (1/2)(\signal{w} + \signal{w}_{\downarrow})= [1,1]^\top$,
which allows the estimate to be updated toward $\signal{x}_{\diamond}$.
We emphasize that this notable advantage comes from the
continuity of the eROWL shrinkage operator $R_{\delta}$.

 \begin{figure}[t!]


\psfrag{w}[Bl][Bl][.8]{$w_2$ ($\tau$)}

\centering
\subfigure[]{
\includegraphics[height=4cm]{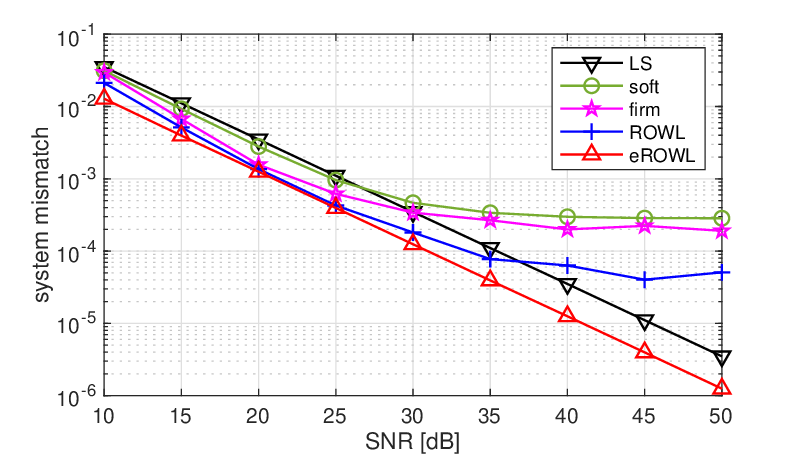}
}\vspace*{-.8em}

\subfigure[]{
\includegraphics[height=4cm]{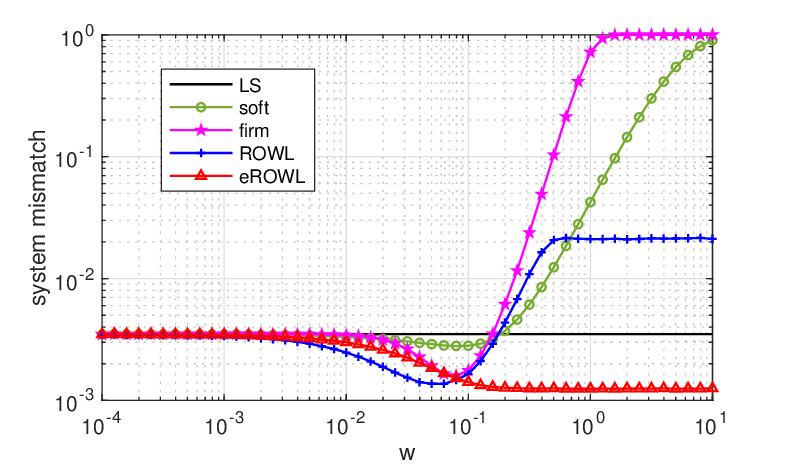}
}\vspace*{-.8em}

  \caption{Comparisons for fixed $\signal{x}_{\diamond}:=[0,1]^\top$.}

 \label{fig:randomA}
 \end{figure}

\subsubsection{ROWL versus eROWL}
\label{subsubsec:rowl_erowl}

We first compare the performance of ROWL and eROWL
with the matrix $\signal{A}$ generated randomly from i.i.d.~$\mathcal{N}(0,1)$
for $M:=8$ and 
with i.i.d.~zero-mean Gaussian $\signal{\epsilon}$.
For reference, soft shrinkage and firm shrinkage
as well as the least squares (LS) estimate, are also tested.
Two types of signal are considered:
(i) $\signal{x}_{\diamond} := [0,1]^{\top}$ (deterministic) and
(ii) $\signal{x}_{\diamond}:=[0,\xi]^\top$ (stochastic) with
$\xi$ generated randomly from $\mathcal{N}(0,1)$.
The results are averaged over 500,000 independent trials.

Figures \ref{fig:randomA} and \ref{fig:raodomx_firm} depict
the results of 
the deterministic case and the stochastic case, respectively.
In Figs.~\ref{fig:randomA}(a) and \ref{fig:raodomx_firm}(a), 
the performance for different SNRs is
plotted, where the weights of ROWL and eROWL are set to $\signal{w}=[0,w_2]^\top$
with the second weight $w_2$ is tuned individually by grid search under SNR 20 dB.
The threshold $\tau$ of soft shrinkage and $\tau_1$ of firm shrinkage
are also tuned individually under SNR 20 dB.
It can be seen that 
eROWL preserves good performance over the whole range, while 
the performance curves of soft, firm, and ROWL saturate as
SNR increases.

 \begin{figure}[t!]
\psfrag{x1}[Bc][Bc][.8]{$x_{\diamond,1}$}

\psfrag{w}[Bl][Bl][.8]{$w_2$ ($\tau$)}

\centering
\subfigure[]{
\includegraphics[height=4cm]{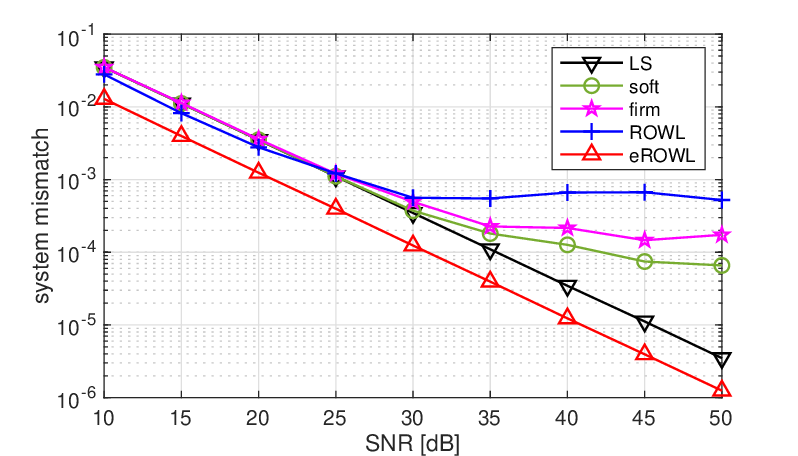}
}
\vspace*{-.8em}

\subfigure[]{
\includegraphics[height=4cm]{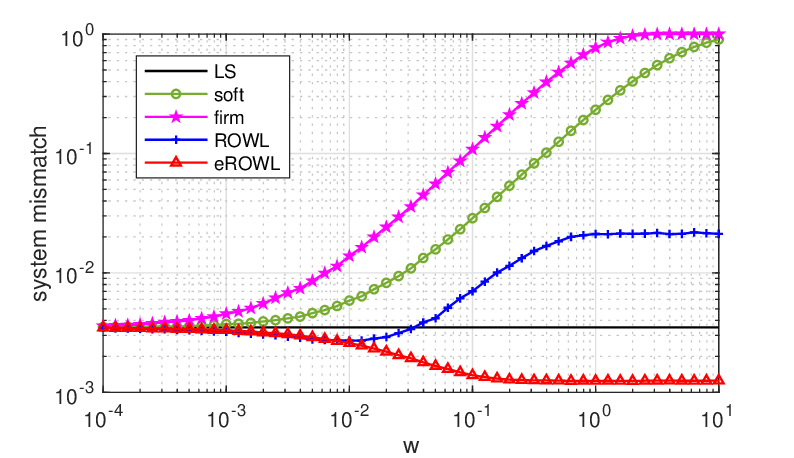}
}
\vspace*{-.8em}

  \caption{
Comparisons for
  $\signal{x}_{\diamond}:=[0,\xi]^\top$ where $\xi$ is random.}
 \label{fig:raodomx_firm}
 \end{figure}

It should also be mentioned that, in the stochastic case,
the performance of firm for low SNRs is nearly identical to that of LS.
This is because the threshold $\tau_1$ must fit the magnitude of the active
component of $\signal{x}_{\diamond}$
but this is difficult in this case as the magnitude changes 
at every trial randomly.
In sharp contrast, eROWL only depends on the ``number'' of active
component (but not on its ``magnitude''), and this is why it performs well.
The performance of ROWL for low SNRs is ony slightly better 
than that of LS
because of its discontinuity (see Section
\ref{sec:simulation}-\ref{subsec:rowl_erowl_simulation}.\ref{subsubsec:illustrative}).
In the deterministic case, although firm and ROWL
achieves comparable performance to eROWL under SNR 20 dB for which the
parameters of each shrinkage operator are tuned,
the performance of those shrinkage operators becomes worse significantly
as SNR becomes apart from 20 dB.

Figures \ref{fig:randomA}(b) and \ref{fig:raodomx_firm}(b)
show the impacts of the tuning parameters $w_2$ (or $\tau$) of
each shrinkage operator for SNR 20 dB.
It clearly shows the stable performance of eROWL which comes from
the continuity of the eROWL shrinkage operator.
This means that $w_2$ is easy to tune and also that
eROWL is expected to be robust against possible environmental changes.
In contrast, the performance of ROWL degrades as $w_2$ becomes larger than
the best value due to its discontinuity.
Note that soft/firm shrinkage for too large threshold level yields 
the zero solution for which the system mismatch is unity.

\section{Conclusion}

We presented the principled way of constructing a continuous relaxation
of a discontinuous shrinkage operator by leveraging
the proximal {\em inclusion} and {\em conversion}
(Theorems \ref{theorem:prox_inclusion} and \ref{theorem:prox_conversion}).
As its specific application, the continuous relaxation of the ROWL
shrinkage was derived.
Numerical examples showed the clear advantages of firm shrinkage and eROWL shrinkage
over hard shrinkage and ROWL shrinkage (the discontinuous counterparts),
demonstrating the efficacy of the continuous relaxation.
A specific situation was presented where the ROWL shrinkage
fails but eROWL shrinkage gives a good approximation of the true
solution.
The simulation results also indicated the potential advantages of eROWL
in terms of simplicity of parameter tuning as well as robustness against
environmental changes.
Although the present study of eROWL is limited to the two dimensional case,
its extension to an arbitrary (finite) dimensional case has been presented in
\cite{okuda_sip24},
where advantages over ROWL have also been shown.
We finally mention that 
the continuous relaxation approach is expected to be useful also for
other nonconvex regularizers such as the one proposed in \cite{wang24}.

\bibliographystyle{IEEEtran}
\bibliography{weaklyconvex}

\appendices
\newcounter{appnum}
\setcounter{appnum}{1}
\setcounter{equation}{0}
\renewcommand{\theequation}{\Alph{appnum}.\arabic{equation}}

\setcounter{figure}{0}
\renewcommand{\thefigure}{\Alph{appnum}.\arabic{figure}}

\section{Derivation of  $R_\delta(\signal{x}_2)$ for $\signal{x}_2\in\mathcal{C}_1$}
\label{appendix:derivation}

The representation of $R_{\delta}$ is visualized in
Fig.~\ref{fig:geometric_Rdelta}.
The halfspaces $\mathcal{H}_{\mbox{{\tiny $\backslash$}}}^-$ and
$\mathcal{H}_{\mbox{{\tiny $\backslash$}}}^+$ share the same boundary which is depicted
by the blue dotted line,
and the boundaries of the hyperslab $\mathcal{S}$ are depicted by 
the black dashed lines.
Figure \ref{fig:geometric_Rdelta}(b), specifically, illustrates the case in which 
$\mathsf{R}(\signal{m})$ touches the $x_1$ and $x_2$ axes.
This situation happens when $\signal{m}$ lies between $[w_1, w_1]^\top$
and $[w_2, w_2]^\top$, and it corresponds to the case of $\signal{x}\in\mathcal{C}_1$.
The point $\signal{p}_1$ is located on the $x_2$ axis, and its inverse image
$\mathsf{R}^{-1}(\signal{p}_1)$ is a set (the line segment in magenta color).
In this case, $\signal{p}_1$ can be expressed as
$\signal{p}_1 =R_{\delta}(\signal{x}_1) = (\signal{x}_1 - \frac{1}{\delta+1}\signal{w}_{\downarrow})_+$,
and it can also be expressed as
$\signal{p}_1 =R_{\delta}(\hat{\signal{x}}_1) =
(\hat{\signal{x}}_1 - \frac{1}{\delta+1}\signal{w}_{\downarrow})_+$.

We rephrase the inclusion relation 
(see Section \ref{sec:application}-\ref{subsec:ordered_weighted_L1})
\begin{equation}
 (\delta +1)\signal{x} \in
\mathsf{R}^{-1}(\signal{p}) + \delta \signal{p},
\label{eq:inclusion_xRinvp}
\end{equation}
which plays a key role in the derivation.
Let $\signal{m}:=\left[\begin{array}{c}
				      m\\m
					   \end{array}\right]$ 
for $m\in [w_1,w_2]$.
Then, $\signal{p}_1$, $\signal{p}_3\in\mathsf{R}(\signal{m})$
in Fig.~\ref{fig:geometric_Rdelta}(b)
can be expressed as
$\signal{p}_1= 
(\signal{m} -\signal{w}_{\downarrow})_+
=\left[\begin{array}{c}
				      0\\m-w_1
					   \end{array}\right]$
and 
$\signal{p}_3=(\signal{m} -\signal{w})_+
= \left[\begin{array}{c}
				      m-w_1\\0
					   \end{array}\right]$,
and thus their convex combination $\signal{p}_2$ is given
by 
\begin{equation}
 \signal{p}_2= \omega \signal{p}_1 + (1-\omega)\signal{p}_3
\label{eq:p2}
\end{equation}
for $\omega\in(0,1)$.
Since \eqref{eq:inclusion_xRinvp} implies that
\begin{equation}
  (\delta +1)\signal{x}_2 = 
\signal{m} + \delta \signal{p}_2,
\end{equation}
we obtain
\begin{align}
m=&~[(\delta+1)(x_1+x_2) + \delta w_1]/(\delta+2), \\
\omega=&~ \frac{(\delta+1)x_2-m}{\delta(m-w_1)}.
\label{eq:omega}
\end{align}
Substituting \eqref{eq:omega} into \eqref{eq:p2} yields
\begin{align}
\hspace*{-.4em} R_{\delta}(\signal{x}_2)=\signal{p}_2
= 
\left[
\begin{array}{c}
m-w_1\\
0
\end{array}
\right]
+
\dfrac{(\delta+1)x_2 - m}{\delta}
\left[
\begin{array}{c}
\!-1 \!\\
1
\end{array}
\right].
\end{align}
The halfspaces $\mathcal{H}_{\mbox{{\tiny $\slash$}} 1}^+$ and $\mathcal{H}_{\mbox{{\tiny $\slash$}} 2}^+$
are derived from the condition $\omega\in (0,1)$,
and $\mathcal{H}_{\mbox{{\tiny $\backslash$}}}^-$ comes from $m \leq w_2$.

The expression of $R_{\delta}(\signal{x})$ for $\signal{x}\in
\mathcal{C}_2$ can be derived analogously by using
$\signal{p}_1= \left[\begin{array}{c}
				      m-w_2\\m-w_1
					   \end{array}\right]$
and 
$\signal{p}_3= \left[\begin{array}{c}
				      m-w_1\\ m-w_2
					   \end{array}\right]$
for $m>w_2$.


 \begin{figure*}[t!]
   \psfrag{0}[Bl][Bl][1]{$\signal{0}$}
   \psfrag{m}[Bl][Bl][1]{$\signal{m}$}
   \psfrag{m=}[Bl][Bl][1]{$\signal{m}=\mathsf{R}^{-1}(\signal{p}_2)$}
   \psfrag{inRinv}[Bl][Bl][1]{$\in \mathsf{R}^{-1}(\signal{p}_1)$}
   \psfrag{mhat}[Bl][Bl][1]{$\hat{\signal{m}}$}
   \psfrag{x}[Br][Br][1]{$\signal{x}_1$}
   \psfrag{xhat}[Bl][Bl][1]{$\hat{\signal{x}}_1$}
   \psfrag{p1}[Br][Br][1]{$\signal{p}_1 =R_{\delta}(\signal{x}_1) =R_{\delta}(\hat{\signal{x}}_1)$}
   \psfrag{p2}[Br][Br][1]{$\signal{p}_2=R_{\delta}(\signal{x}_2)$}
   \psfrag{p3}[Br][Br][1]{$\signal{p}_3$}
   \psfrag{x2vec}[Bl][Bl][1]{$\signal{x}_2$}
   \psfrag{Rm}[Br][Br][1]{$\mathsf{R}(\signal{m})$}
   \psfrag{Rinvp1}[Bl][Bl][1]{$\mathsf{R}^{-1}(\signal{p}_1)$}
   \psfrag{C1}[Bc][Bc][1]{$\mathcal{C}_1$}
   \psfrag{C2}[Bc][Bc][1]{$\mathcal{C}_2$}
   \psfrag{M}[Bc][Bc][1]{$\mathcal{M}$}
   \psfrag{-w}[Bc][Bc][1]{$-\signal{w}_{\downarrow}$}
   \psfrag{Tm}[Bc][Bc][1]{$\mathsf{R}(\signal{m})$}
   \psfrag{mmin}[Bl][Bl][1]{$\left[\begin{array}{c}
				      w_1\\ w_1
					   \end{array}\right]$}
   \psfrag{mmax}[Bl][Bl][1]{$\left[\begin{array}{c}
				      w_2\\ w_2
					   \end{array}\right]$}
   \psfrag{w1delta+1}[Bl][Bl][1]{$\dfrac{1}{\delta+1}\left[\begin{array}{c}
				      w_1\\ w_1
					   \end{array}\right]$}

   \psfrag{=}[Bl][Bl][1]{$= \mathsf{R}^{-1}(\hat{\signal{p}})$}

   \psfrag{x1}[Bl][Bl][1]{$x_1$}
   \psfrag{x2}[Br][Br][1]{$x_2$}
   \psfrag{p}[Bl][Bl][1]{\hspace*{-4em}$\signal{p} = R_{\delta}(\signal{x})$}
   \psfrag{xt}[Br][Br][1]{$\hat{\signal{x}}$}
   \psfrag{pt}[Br][Br][1]{$\hat{\signal{p}}$}
   \psfrag{pt}[Br][Br][1]{$\hat{\signal{p}} = R_{\delta}(\hat{\signal{x}})$}
   \psfrag{distrho}[Bl][Bl][1]{$\dfrac{(w_1+w_2)\delta}{\sqrt{2}(\delta+1)}$}
   \psfrag{dist1}[Bl][Bl][1]{$\dfrac{w_1+w_2}{\sqrt{2}(\delta+1)}$}
   \psfrag{dist2rho}[Br][Br][1]{$\dfrac{(w_2- w_1)\delta}{\sqrt{2}(\delta+1)}$}
   \psfrag{dist21}[Br][Br][1]{$\dfrac{w_2 - w_1}{\sqrt{2}(\delta+1)}$}
   \psfrag{w2}[Br][Br][1]{$w_2-w_1$}
   \psfrag{w2delta}[Br][Br][1]{$\dfrac{(w_2-w_1)\delta}{\delta+1}$}
 \centering

\subfigure[]{
\hspace*{.4cm}\includegraphics[height=9cm]{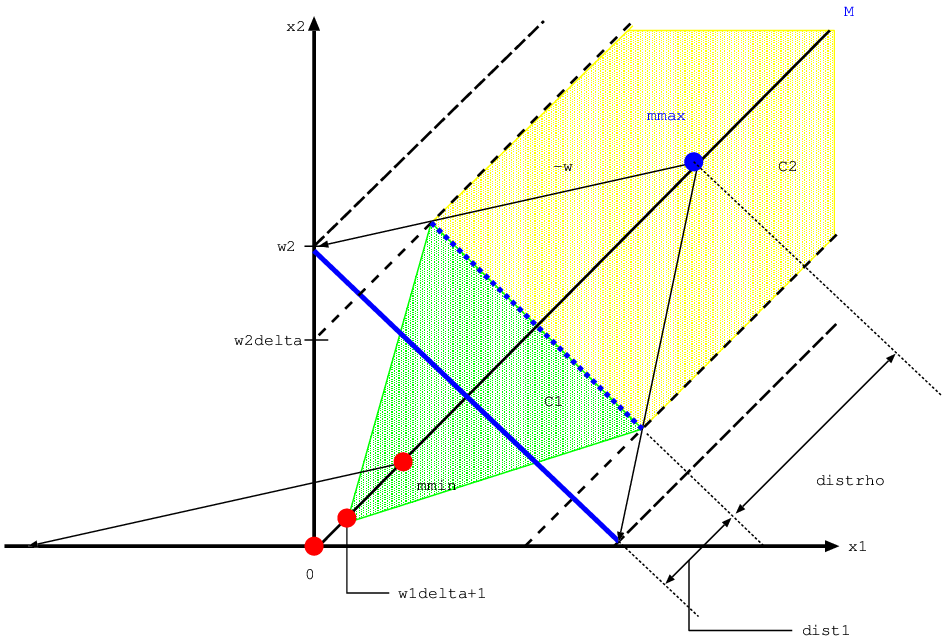}
}

\subfigure[]{
\hspace*{.4cm}\includegraphics[height=9cm]{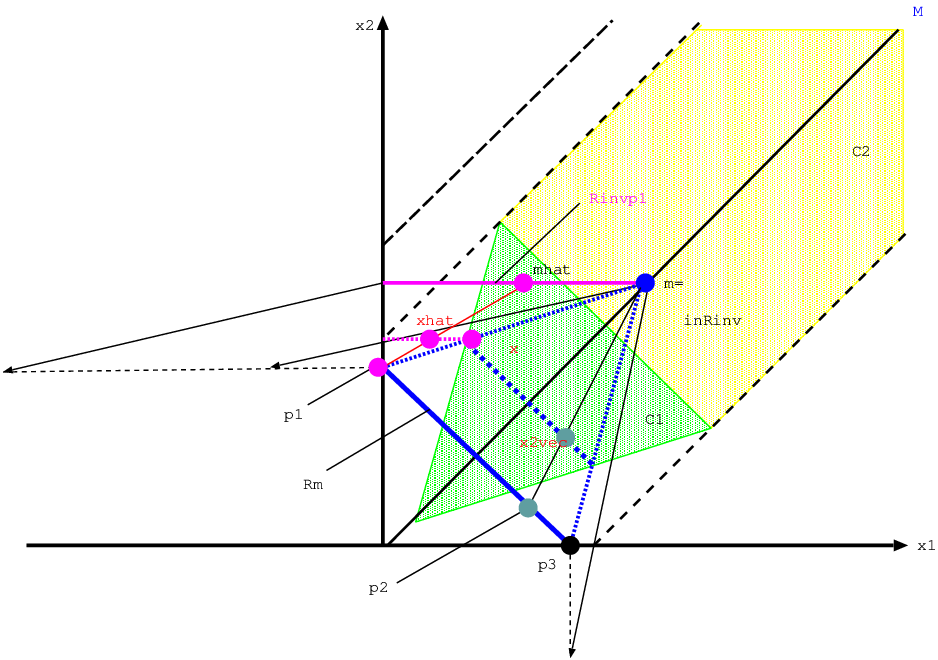}
}

  \caption{Geometric interpretations of the operator $R_{\delta}$.}
\vspace*{-.5em}
 \label{fig:geometric_Rdelta}
 \end{figure*}

\setcounter{appnum}{2}
\setcounter{equation}{0}
\renewcommand{\theequation}{\Alph{appnum}.\arabic{equation}}

\setcounter{figure}{0}
\renewcommand{\thefigure}{\Alph{appnum}.\arabic{figure}}

\section{Parameters of Firm Shrinkage}
\label{appendix:firm_parameter}

By \cite[Theorem 2]{yukawa_molgrad24},
the iterate given in \eqref{eq:pfbs} converges to
the minimizer of
$\mu f + \varphi$ under the conditions
$\beta\in((\kappa-\rho)/(\kappa+\rho),1)$ and
$\mu\in [(1-\beta)/\rho, (1+\beta)/\kappa)$
(see Fact \ref{fact:weaklyconvex_necsuffcondition} for the relation
between $\varphi$ and the given operator $T$).
Since ${\rm firm}_{\tau_1,\tau_2}$ is $\tau_2/(\tau_2-\tau_1)$-Lipschitz
continuous, we have
$\beta^{-1} = \tau_2/(\tau_2-\tau_1)
\Leftrightarrow
\beta=1-\tau_1/\tau_2$.
This leads to the condition 
$1-\tau_1/\tau_2>(\kappa-\rho)/(\kappa+\rho)
\Leftrightarrow 
\tau_2> \tau_1 (\kappa+\rho)/(2\rho)$.
To satisfy this inequality, we set $\tau_2:= \tau_1 (\kappa+\rho)/((2 -
\varepsilon)\rho)$ for a small constant $\varepsilon\in(0,2)$.
In this case, we have $\beta= (\kappa-(1-\varepsilon)\rho)/(\kappa+\rho)$,
and 
setting the step size $\mu$ to its lower bound gives
 $\mu= (1-\beta)/\rho \Leftrightarrow (2-\varepsilon)/(\kappa+\rho)$.

\vfill\pagebreak

\end{document}